\documentclass[11pt,reqno]{amsart}
\usepackage[colorlinks=true,allcolors=blue,backref=page]{hyperref}
\usepackage{color}
\usepackage{amsmath, amssymb, amsthm}

\usepackage{mathrsfs}
\usepackage{mathtools}
\usepackage[noabbrev,capitalize,nameinlink]{cleveref}
\crefname{equation}{}{}
\usepackage{fullpage} 
\usepackage[noadjust]{cite}
\usepackage{graphics}
\usepackage{pifont}
\usepackage{tikz}
\usepackage{tikz-cd}
\usepackage{bbm}
\usepackage[T1]{fontenc}

\usetikzlibrary{arrows.meta}

\usepackage{environ}
\usepackage{framed}
\usepackage{url}
\usepackage[linesnumbered,ruled,vlined]{algorithm2e}
\usepackage[noend]{algpseudocode}
\usepackage[labelfont=bf]{caption}
\usepackage{cite}
\usepackage{framed}
\usepackage[framemethod=tikz]{mdframed}
\usepackage{appendix}
\usepackage{graphicx}
\usepackage[textsize=tiny]{todonotes}
\usepackage{tcolorbox}
\usepackage{enumerate}
\usepackage[shortlabels]{enumitem}
\allowdisplaybreaks[1]

\apptocmd{\sloppy}{\hbadness 10000\relax}{}{} % magic BibTeX spacing fix

\crefname{algocf}{Algorithm}{Algorithms}

\crefname{equation}{}{} %remove ``Equation''
\crefname{conjecture}{Conjecture}{Conjectures} %add ``Conjecture''
\AtBeginEnvironment{appendices}{\crefalias{section}{appendix}} %appendices

\crefformat{enumi}{#2#1#3}
\crefrangeformat{enumi}{#3#1#4 to~#5#2#6}
\crefmultiformat{enumi}{#2#1#3}%
{ and~#2#1#3}{, #2#1#3}{ and~#2#1#3}

\usepackage[final]{showkeys} %add in 'final' into parameter to remove showkeys

\colorlet{refkey}{orange!20}
\colorlet{labelkey}{blue!30}

\crefname{algocf}{Algorithm}{Algorithms}

\numberwithin{equation}{section}
\newtheorem{theorem}{Theorem}[section]
\newtheorem{proposition}[theorem]{Proposition}
\newtheorem{lemma}[theorem]{Lemma}

\crefname{claim}{Claim}{Claims}

\newtheorem*{question*}{Question}

\theoremstyle{definition}
\newtheorem{definition}[theorem]{Definition}

\newtheorem*{definition*}{Definition}

\theoremstyle{remark}
\newtheorem*{remark}{Remark}
\newtheorem*{remarks}{Remarks}

\newcommand{\snorm}[1]{\lVert#1\rVert}

\newcommand{\mb}{\mathbb}
\newcommand{\mbf}{\mathbf}

\newcommand{\mc}{\mathscr}

\newcommand{\wh}{\widehat}
\newcommand{\wt}{\widetilde}

\newcommand{\eps}{\varepsilon}

\renewcommand{\le}{\leqslant}
\renewcommand{\ge}{\geqslant}

\newcommand\Z{\mathbf{Z}}
\newcommand\Q{\mathbf{Q}}
\newcommand\C{\mathbf{C}}
\newcommand\R{\mathbf{R}}
\newcommand\N{\mathbf{N}}

\def\sml{\operatorname{sml}}
\def\lrg{\operatorname{lrg}}

\newcommand\diam{\operatorname{diam}}

\allowdisplaybreaks

\newcommand{\md}[1]{\ensuremath{(\operatorname{mod}\, #1)}}
\newcommand{\mdsub}[1]{\ensuremath{(\mbox{\scriptsize mod}\, #1)}}

\usepackage{graphicx}

\title{Bounds for monochromatic solutions to $\{x+y,xy\}$}
\author[A1]{Ben Green}
\address{Mathematical Institute, Andrew Wiles Building, Radcliffe Observatory Quarter, Woodstock Rd, Oxford OX2 6QW, UK}
\email{ben.green@maths.ox.ac.uk}

\author[A2]{Mehtaab Sawhney}
\address{Department of Mathematics, Columbia University, New York, NY 10027}
\email{m.sawhney@columbia.edu}
\begin{document}

\begin{abstract}
Let $r$ be a sufficiently large positive integer, and let $N \ge \exp\exp(r^{50})$. Then any $r$-colouring of $[N]$ contains a monochromatic copy of $\{x+y,xy\}$ with $x > y > 2$.
\end{abstract}

\maketitle
\setcounter{tocdepth}{1}
\tableofcontents

\section{Introduction}

The key result in this work is an effective bound for $r$-colourings of the natural numbers $\N$ containing a monochromatic copy of $\{x+y,xy\}$.
\begin{theorem}\label{thm:main}
There is a constant $r_0$ such that the following holds. Let $r \ge r_0$ be an integer and let $N \ge \exp\exp(r^{50})$. Then any $r$-colouring of $[N] :=\{1,\ldots,N\}$ contains a monochromatic copy of $\{x+y,xy\}$ with $x > y > 2$.
\end{theorem}
\begin{remarks} The constant $r_0$ is effectively computable. Furthermore, minor tweaks to the numerics in our arguments would allow one to replace $50$ with a slightly smaller constant. However, to obtain a `small' constant (less than 10, say) would appear to require new ideas. Finally we make no effort to compute an actual value of $r_0$. Due to arguments regarding the possible existence of a Siegel zero in \cref{appC} (among other reasons), to do so would be rather painful.
\end{remarks}
In the other direction, for all $r$ there is an $r$-colouring of $[N]$ with no monochromatic $\{x + y, xy\}$ with $x > y > 2$ when $N = \frac{1}{2}(3^{r} + 7)$, and therefore \cref{thm:main} is at most one logarithm from the optimal result. To obtain such a colouring, use colour $i$ for $[a_i, a_{i + 1})$ where $a_i := \frac{1}{2}(3^i + 9)$, $i = 0,\dots, r - 1$, and any colour for $\{1,2,3,4\}$. The point here is that $a_{i + 1} = 3(a_i - 3)$ and so if $x + y \in [a_i, a_{i + 1})$ with $x, y \ge 3$ then $xy \ge a_{i+1}$. We never have $x+y$ or $xy \in \{1,2,3,4\}$.

We remark that obtaining effective bounds for the pattern $\{x+y,xy\}$ has been raised by both the first author \cite[Problem~22]{GreOp} and by Richter \cite[Question~7.2]{Ric25}. 

\subsection{Previous results}
\cref{thm:main} guarantees the existence of infinitely many pairs $\{x+y,xy\}$ given a fixed $r$-colouring of $\N$. To see this, suppose that we have found $d$ such monochromatic pairs $\{x_i + y_i, x_i y_i\}$, $i = 1,\dots, d$. We modify our colouring of $\N$ to an $(r + 2d)$-colouring in which $x_1,\dots, x_d, y_1,\dots, y_d$ are given distinct colours, different to the original $r$, and then use \cref{thm:main} to find a further pair $\{x_{d+1} + y_{d+1}, x_{d+1}y_{d+1}\}$. (Alternatively one may observe that our proof of \cref{thm:main} may be trivially modified to give many monochromatic pairs as $N \rightarrow \infty$ for a fixed value of $r$.) 

This existential statement was first proven in a celebrated paper of Moreira \cite{Mor17}; furthermore Moreira in fact guarantees a monochromatic pattern of the form $\{x,x+y,xy\}$. This result represents substantial progress towards Hindman's conjecture that any $r$-colouring of $\N$ contains a monochromatic copy of $\{x,y,x+y,xy\}$. Recently there has been further important progress towards Hindman's conjecture in various settings. Bowen \cite{Bow25} has proven that any $2$-colouring of $\N$ contains infinitely many copies of $\{x,y,x+y,xy\}$. Bowen and Sabok \cite{BS24} have proven that any $r$-colouring of $\Q^{\neq 0}$ contains a copy of $\{x,y,x+y,xy\}$ and Alweiss  \cite{Alw23} extended this to patterns of the form $\{\sum_{i\in S}x_i, \prod_{i\in S}x_i\}$ where $S\subseteq [k]$ ranges over all nontrivial subsets. Additionally Alweiss \cite{Alw24} has given an alternate proof of the result of Moreira. However even when restricting to $\{x+y,xy\}$ the proofs of Moreira and Alweiss give at least tower--type bounds due to highly recursive Ramsey type arguments. We remark that while the main argument of Moreira is purely qualitative, he indicates in \cite[Section 5]{Mor17} a variant argument using van der Waerden's theorem (or Szemer\'edi's theorem) which does give explicit finite bounds when used with appropriate bounds for Szemer\'edi's theorem due to Gowers \cite{Gow01}. 

Recently, Richter \cite{Ric25} provided a quite different, more analytic, proof of Moreira's result about $\{x + y, xy\}$. The argument of Richter is quite infinitary in flavour and gives no bounds. However, as will be discussed shortly, our methods in this paper are very strongly influenced by those of Richter.

One may additionally compare \cref{thm:main} with bounds for certain Schur-type equations. For instance, for the configuration $\{x,y,x+y\}$, bounds of the form $\exp(r^{O(1)})$ are known due to work of Cwalina and Schoen \cite{CS17}. Note that this (by restricting to powers of $2$) gives an essentially double-exponential bound for $\{x,y,xy\}$. Furthermore for more general linear systems $A$, bounds of the form $\exp\exp(r^{O_A(1)})$ are proven in generality by Sanders \cite{San20}, and good control on the implicit constant $O_A(1)$ for many systems may be found in work of Chapman and Prendiville \cite{CP20}. 

\subsection{Proof outline}\label{outline}
Our work draws heavily on recent beautiful work of Richter \cite{Ric25}; many of the ideas presented in this section are drawn from this work. 

Logarithmic averages play a central role, so we define these before turning to an outline of the proof. If $\mathcal{N}$ is a finite set of positive integers and if $f : \mathcal{N} \rightarrow \C$ is a function, we write 
\[ \mb{E}_{n \in \mathcal{N}}^{\log} f(n) := \frac{\sum_{n \in \mathcal{N}} f(n)/n }{\sum_{n \in \mathcal{N}} 1/n}.\] We write $\mb{E}_{n_1 \in \mathcal{N}_1, n_2 \in \mathcal{N}_2}^{\log}$ as a shorthand for $\mb{E}_{n_1 \in \mathcal{N}_1}^{\log} \mb{E}^{\log}_{n_2 \in \mathcal{N}_2}$ (and similarly for higher iterates). We will often use this notation when $\mathcal{N} = [N] = \{1,\dots, N\}$.

Suppose now that $[N] = A_1 \cup \cdots \cup A_r$ is an $r$-colouring of $[N] = \{1,\dots, N\}$ in which we seek to find a monochromatic pair $x + y, xy$. The colour class in which this pair will be found is identified right at the very start of the proof. We take $B_0$ to be a fixed set of $r^{O(1)}$ `highly divisible' numbers; the precise set we take is $B_0 := \{V^{4^i} : i = 1,2,\dots, r^{C_1}\}$, where $V = (r^{C_2})!$ for appropriate constants $C_1,C_2$. By the pigeonhole principle there is some $A = A_{\ell}$ which contains many multiples of elements of $B_0$ in the sense that $\mb{E}_{n \in [N]}^{\log} 1_A(b n) \gg 1/r$ for at least $\gg r^{C_1 - 1}$ elements $b \in B_0$. We will find the desired configuration $\{x + y, xy\}$ in this colour class, which we fix for the rest of the argument.

The next key idea, which follows \cite{Ric25} very closely, is to locate a `rich' set of pairs $\{x , xy\}$ in $A$. This is done using a variant of arguments of Ahlswede, Khachatrian and S{\'{a}}rk{\"{o}}zy \cite{AKS99} and Davenport and Erd\H{o}s \cite{DE36}. This argument involves the choice of various auxiliary sets of primes (for details see \cref{sec61}) and a key component is Elliott's inequality from multiplicative number theory (given in \cref{lem:Elliot} in the form we shall need). The output of this argument is many instances of the inequality
\begin{equation}\label{ramsey-outcome} \mb{E}_{n \in [N],p_{1}\in \mc{P}_{1},\dots, p_{k}\in \mc{P}_{k}}^{\log} 1_A(bn) 1_A(b' p_1 \cdots p_k n) \gg r^{-O(1)} \end{equation} for some fixed $b \in B_0$ and many $b' \in B_0$ with $b < b'$ and associated $k$ where $2 \le k \ll r^{O(1)}$, and where the sets $\mc{P}_i$ of primes can be chosen at many different scales. (The precise statement we are sketching here may be found at \cref{main-ramsey-statement}.) This provides the aforementioned rich source of configurations $\{x, xy\}$, here with $x := bn$ and $y := \frac{b'}{b} p_1 \cdots p_k$, 

The main business of the proof is a kind of deformation of the patterns $x,xy$ to the desired $x+y,xy$. To describe how this works, fix an instance of \cref{ramsey-outcome} (that is, fix $b'$ and the sets $\mc{P}_i$ of primes). Set $f(n) := 1_A(bn)$. We will then consider two `projections' $\Pi^{\sml} f$ and $\Pi^{\lrg} f$, both of which average over progressions. They are defined by
\[ \Pi^{\sml} f(n) := \mb{E}_{h,h' \in [H]} f(n + q(h - h')) \qquad  \mbox{and} \qquad \Pi^{\lrg} f(n) := \mb{E}_{h,h' \in [\tilde H]} f(n + \tilde q(h - h')) \] where here $q  \mid \tilde q$ and $H > \tilde H$. (The actual choice of parameters depends on the scale of the sets of primes $\mc{P}_i$; the details are given at \cref{pi-pm}). One should think of $q, \tilde q$ as being bounded in terms of $r$, whereas the lengths $H, \tilde H$ grow with $N$. 

The small projection $\Pi^{\sml}$ is chosen so that we may run the following argument, starting from \cref{ramsey-outcome}. First, via a kind of maximal function argument, we replace \cref{ramsey-outcome} by

\begin{equation}\label{ramsey-outcome-2} \mb{E}_{n \in [N],p_{1}\in \mc{P}_{1},\dots, p_{k}\in \mc{P}_{k}}^{\log} \Pi^{\sml} f(n) 1_A(b' p_1 \cdots p_k n) \gg r^{-O(1)} .\end{equation} Details of this argument may be found in \cref{maximal}.

Then, we use the almost-periodicity property $\Pi^{\sml} f(n) \approx \Pi^{\sml} f(n + \frac{b'}{b^2} p_1 \cdots p_k)$ to replace \cref{ramsey-outcome-2} by
\begin{equation}\label{ramsey-outcome-3} \mb{E}_{n \in [N],p_{1}\in \mc{P}_{1},\dots, p_{k}\in \mc{P}_{k}}^{\log} \Pi^{\sml} f\big(n + \frac{b'}{b^2} p_1 \cdots p_k\big) 1_A(b' p_1 \cdots p_k n) \gg r^{-O(1)} .\end{equation} (note here that $b'/b^2$ is an integer by the highly divisible nature of the set $B_0$). In order for this almost-periodicity property to hold, the small projection $\Pi^{\sml}$ must be chosen appropriately: $q$ must divide $\frac{b'}{b^2} p_1 \cdots p_k$ and $H$ must be sufficiently long.

Leaving \cref{ramsey-outcome-3} aside for the moment, the technical heart of the proof is then an argument to the effect that (for an appropriate choice of the large projection $\Pi^{\lrg}$) we have
\begin{align} \nonumber \mb{E}_{n \in [N],p_{1}\in \mc{P}_{1},\dots, p_{k}\in \mc{P}_{k}}^{\log} & \Pi^{\lrg} f\big(n + \frac{b'}{b^2} p_1 \cdots p_k\big) 1_A(b' p_1 \cdots p_k n) \\ & \approx \mb{E}_{n \in [N],p_{1}\in \mc{P}_{1},\dots, p_{k}\in \mc{P}_{k}}^{\log} f\big(n + \frac{b'}{b^2} p_1 \cdots p_k\big) 1_A(b' p_1 \cdots p_kn). \label{from-inverse-proj}\end{align}
Supposing that this has been established, imagine that we additionally have 
\begin{equation} \label{small-large} \Pi^{\sml} f \approx \Pi^{\lrg} f \end{equation} (in an $\ell^2$ sense). Combining \cref{ramsey-outcome-3,from-inverse-proj,small-large} then gives, assuming the various uses of $\approx$ work in our favour, that 
\[ \mb{E}_{n \in [N],p_{1}\in \mc{P}_{1},\dots, p_{k}\in \mc{P}_{k}}^{\log} f\big(n + \frac{b'}{b^2} p_1 \cdots p_k\big) 1_A(b' p_1 \cdots p_k n) \gg r^{-O(1)} .\] Recalling that $f(n) = 1_A(bn)$, it then follows that for some choice of $n$ and $p_1,\dots, p_k$ we have $bn + \frac{b'}{b} p_1 \cdots p_k, b' p_1 \cdots p_k n \in A$. This is the desired configuration $\{ x + y, xy\}$, with $x = bn$ and $y = \frac{b'}{b} p_1 \cdots p_k $.

Whilst \cref{small-large} will not be true in general (the projections $\Pi^{\sml},\Pi^{\lrg}$ are quite different in scale), an `energy-chaining' or arithmetic regularity type of argument can be used to show that \cref{small-large} does hold for at least one scale of primes $\mc{P}_1,\dots, \mc{P}_k$. This part of the argument can be thought of as a quantitative version of the existence of projections in Hilbert space, specifically of the decomposition into locally aperiodic and locally quasiperiodic functions which is important in Richter's work. This connection between existence of projections in Hilbert space and regularity lemmas is by now well established; see e.g. \cite[Section~2]{Tao07}.

The remaining part of the argument is then to justify \cref{from-inverse-proj}. This is done via a general study of averages
\begin{equation}\label{mainf1f2avg} \mb{E}_{n \in [N],p_1 \in \mc{P}_1,\dots, p_k \in \mc{P}_k}^{\log} f_1(n + \lambda p_1 \cdots p_k) f_2(n p_1 \cdots p_k),\end{equation} where $\lambda = b'/b^2$ in our setting. Here, we consider arbitrary $1$-bounded functions $f_1, f_2$, and the key question of interest is the `inverse question' of what can be said if \cref{mainf1f2avg} is at least $\delta$ in magnitude for some $\delta > 0$. Our main result on this topic, \cref{main-sec3}, is an inverse theorem for this question. It concludes that under such a hypothesis (and with suitable assumptions on the sets $\mc{P}_i$ of primes) the function $f_1$ is biased along progressions to some modulus $\tilde q = \lambda \lfloor \delta^{-C}\rfloor!$ and length $H$ comparable (in logarithmic scale) to the largest of the primes $\mc{P}_i$. The statement \cref{from-inverse-proj} follows very quickly from this inverse theorem (see \cref{lem:proj-check} for the argument).

This inverse theorem, \cref{main-sec3}, is the most novel part of our paper. Whilst it is in a sense a quantitative, finitary version of \cite[Theorem 3.5]{Ric25}, it is not a direct translation of that result, which would appear to be far too weak for our purposes. The key difference when unwinding the argument in \cite{Ric25} in finitary language is that the latter finds bias along progressions with size depending on $\mc{P}_i$ while ours depends only on $\delta$. The proof of \cref{main-sec3} is lengthy, and involves a Fourier analytic argument combined with Cauchy--Schwarz man{\oe}uvres inspired by certain ``concatenation'' results in the additive combinatorics literature, for instance \cite{PP24,Pel20}. Ultimately it is these concatenation ideas which eliminate the dependence on $\mc{P}_i$.  Key further ingredients are:
\begin{itemize}
    \item Quantitative diophantine approximation results (\cref{vino-lemma});
    \item `Log-free' exponential sum estimates for certain arithmetic sets, specifically sets $\mc{P}' = \{ p_2 \cdots p_k : p_2 \in I_2,\dots, p_k \in I_k\}$ of `almost primes', as well as the sets of squares of the elements of such sets (\cref{sec3});
    \item Construction of a majorant for the primes with a certain Fourier decomposition (\cref{sec-fourier-decomp}), in order to avoid the constant $r_0$ in our main result being ineffective due to possible Siegel zeros.
\end{itemize}

\subsection{Acknowledgments}
BG is supported by Simons Investigator Award 376201. This research was conducted during the period MS served as a Clay Research Fellow. 

\subsection{Notation}\label{notation-sec}
At various points, for brevity it will be expedient to use the following notation. If $f : \Z \rightarrow \C$ is a function and if $h,h' \in \Z$, we write $\Delta_{(h,h')}f(x) := f(x + h) \overline{f (x + h')}$. If $\lambda$ is some further integer parameter, by $\Delta_{\lambda(h, h')} f$ we mean $\Delta_{(\lambda h, \lambda h')}f$.

By a \emph{dyadic interval} we mean any subset of $\N$ of the form $\{n : Y \le n < 2Y\}$. We will occasionally abuse notation by writing $[H]$ when we really mean $[\lfloor H \rfloor]$, for some $H \in \R_{\ge 1}$.

When we say that a parameter (for instance $\delta$) is `sufficiently small' we mean that $\delta \le \delta_0$ for some absolute $\delta_0$ which we do not explicitly specify, and analogously if we say that $N$ is `sufficiently large' we mean that $N \ge N_0$ for some absolute constant $N_0$. It is important to remark that $\delta_0, N_0$ are absolute and do not depend on the number of colours $r$ (otherwise our results would have little content). Throughout the paper the letter $N$ will always denote a sufficiently large integer parameter.

We write $(x,y)$ for the greatest common divisor of $x,y$ and $[x,y]$ for the lowest common multiple.

\section{Diophantine sets and averages}\label{dio-sec}

The purpose of this section is to bound certain averages that will appear in the arguments of the next section, where our key technical result is established. The averages in question will be of the form 
\[ \mb{E}_{n \in [N]}^{\log} \mb{E}_{s \in S, t,t' \le T} f(n + ts) \overline{f(n + t' s)} = \mb{E}_{n \in [N]}^{\log} \mb{E}_{s \in S, t,t'\in T} \Delta_{t(s,s')} f(n), \] where $S \subset \N$ is contained in some dyadic interval, or the analogous average with $\mb{E}_{n \in [N]}$ in place of the logarithmic average. The main result of the section is \cref{lem:input-concat} below.

In our applications the set $S$ will have a useful arithmetic property, namely that it satisfies a `log-free Weyl-type estimate'. The precise definition we will use is the following.

\begin{definition}\label{dioph-def}
Let $L, L', D$ be parameters. Let $S$ be a set of integers. Suppose that whenever $\delta \in (0, \frac{1}{2})$ and $| \mb{E}_{s \in S} e(\theta s)| \ge \delta$, then there is some natural number $q$, $q \le (L'/\delta)^{L}$, such that $\Vert q \theta \Vert_{\R/\Z} \le (L'/\delta)^L/D$. Then we say that $S$ is $(L, L', D)$-diophantine.
    \end{definition}
    \begin{remarks} Note that the definition is invariant under translation of $S$. In applications the parameter $D$ will be comparable to the diameter of $S$, but it is convenient not to simply set $D := \diam(S)$, since this would lead to unnecessary estimations of the diameter of $S$ in some situations. Being diophantine with $D \asymp \diam(S)$ (for some $L, L'$) is a common property of sets of integers. For instance, (the log-free variant of) Weyl's inequality asserts that the set of $j$th powers in $[D]$ is $(L, L',D)$-diophantine with appropriate parameters $L, L' \ll_j 1$; the set of $j$th powers of primes in $[D]$ is also $(L, L',D)$-diophantine for some $L, L' \ll_j 1$. In fact, we will use the latter fact in our argument; for the proof see \cref{log-free-weyl-2}. 
\end{remarks}

Before turning to the statement and proof of the main results, we isolate the following lemma, which is of a standard type in the analysis of exponential sums. A proof of this particular variant may be found in \cite[Lemma C.1]{Gre25} (we have changed some dummy variables to avoid conflicts with the present paper).
\begin{lemma}\label{vino-lemma}
 Suppose that $\alpha \in \R$ and that $T \ge 1$ is an integer. Suppose that $\delta_1, \delta_2$ are positive real numbers satisfying $\delta_2 \ge 32 \delta_1$, and suppose that there are at least $\delta_2 T$ elements $t \in [T]$ for which $\Vert \alpha t \Vert_{\R/\Z} \le \delta_1$. Suppose that $T \ge 16/\delta_2$. Then there is some positive integer $q \le 16/\delta_2$ such that $\Vert \alpha q \Vert_{\R/\Z} \le \delta_1\delta_2^{-1} T^{-1}$.
\end{lemma}

We next give the definition of certain norms describing bias of functions along arithmetic progressions.

\begin{definition}\label{gp-local-def}
Let $f : \Z \rightarrow \C$ be a function. Let $q \in \N$ and $H \in \N$ be parameters. Set 
\begin{equation}\label{log-u1-def} \Vert f \Vert_{U^1_{\log}[N; q,H]}^2 := \mb{E}_{n \in [N]}^{\log} \big| \mb{E}_{h \in [H]} f(n + hq) \big|^2 = \mb{E}_{n \in [N]}^{\log} \mb{E}_{h,h' \in [H]}\Delta_{q(h,h')} f(n) \end{equation} and
\begin{equation}\label{non-log-u1-def} \Vert f \Vert_{U^1[N; q,H]}^2 := \mb{E}_{n \in [N]} \big| \mb{E}_{h \in [H]} f(n + hq) \big|^2 = \mb{E}_{n \in [N]} \mb{E}_{h,h' \in [H]}\Delta_{q(h,h')} f(n). \end{equation}
\end{definition}
The logarithmic norm \cref{log-u1-def} will play the more prominent role in our analysis, with the uniform norm \cref{non-log-u1-def} being relegated to a more modest technical role in \cref{lem:input-concat-2-iter}. We record that, roughly speaking, we have $\Vert f \Vert_{U^1_{\log}[N; q,H]} \lessapprox  \Vert f \Vert_{U^1_{\log}[N; \tilde q,\tilde H]}$ if $q \mid \tilde q$ and that $\tilde H \tilde q < H q$ (for a precise statement, see \cref{gp-compar}). In particular for fixed $q$ the information that $\Vert f \Vert_{U^1_{\log}[N; q,H]}$ is large becomes weaker as $H$ becomes smaller. We are now ready for the first main result of the section, which could potentially have other applications.

\begin{lemma}\label{lem:input-concat}
Let $\delta$ be a sufficiently small positive parameter and $L,L',D\ge 1$. Let $S \subset \Z$ be $(L,L', D)$-diophantine with $S \subset [-4D, 4D]$, and let $T \in \N$ be a parameter. Suppose that $D, T \ge (L'/\delta)^{8L}$ and that $\frac{\log TD}{\log N} \le (\delta/L')^{50L}$. Let $H$ be any positive integer with $H \le (\delta/L')^{50L} TD$. Let $f : \N \rightarrow \C$ be $1$-bounded and suppose that we have
\begin{equation}\label{lem26-assump} \mb{E}_{n \in [N]}^{\log}\mb{E}_{t,t' \in [T]}\mb{E}_{s\in S}  f(n + ts) \overline{f(n + t's)} \ge \delta.\end{equation}
Then there exists $q \in \N$, $q\le (L'/\delta)^{8L}$, such that $\Vert f \Vert_{U^1_{\log}[N; q,H]} \ge (\delta/L')^{25L}$. 
\end{lemma}
\begin{remark}
Note here that $q$ may depend on $f$, but we are free to specify $H$ subject to the stated upper bound condition.
\end{remark}
\begin{proof}  Throughout the proof we assume that $\delta_0$ is sufficiently small without further comment. The proof is Fourier-analytic; closely related arguments have appeared as base cases for various `concatenation' results (see e.g. \cite[Lemma~5.3]{PP24} or \cite[Lemma~5.4]{Pel20}). By \cref{log-avg-shift} applied with $h = (t - t') s$ we have
\[\mb{E}_{n \in [N]}^{\log}f(n)\mb{E}_{t,t'\in [T]}\mb{E}_{s\in S}\overline{f(n + (t-t')s)} \ge \delta/2, \] which for brevity we write
\[\mb{E}_{n \in [N]}^{\log}f(n)\mb{E}_{u \in [T] - [T]}\mb{E}_{s\in S}\overline{f(n + u s)} \ge \delta/2, \] with the understanding that $[T] - [T]$ is considered with multiplicity.
By Cauchy--Schwarz this gives that 
\[\mb{E}_{n \in [N]}^{\log}\mb{E}_{u,u'\in [T]-[T]}\mb{E}_{s,s'\in S} f(n + us)\overline{f(n+u's')}\ge \delta^{2}/4.\]
By a further application of \cref{log-avg-shift}, followed by the triangle inequality, we have that 
\[\mb{E}_{n \in [N]}^{\log}\Big|\mb{E}_{h\in [TD]-[TD]}\mb{E}_{u,u'\in [T]-[T]}\mb{E}_{s,s'\in S} f(n +h+ us)\overline{f(n+h + u's')}\Big|\ge \delta^{2}/8.\] Denote
\[ \mathcal{N}_0 :=\big\{ n \in [N] : \big|\mb{E}_{h\in [TD]-[TD]}\mb{E}_{u,u'\in [T]-[T]}\mb{E}_{s,s'\in S} f(n +h+ us)\overline{f(n+h + u's')}\big|\ge  \delta^{2}/16\big\}.\] By a simple averaging argument we have \begin{equation}\label{2point1-star} \mb{E}^{\log}_{n \in [N]} 1_{\mathcal{N}_0} (n) \ge \delta^2/16.\end{equation} For the time being, let $n \in \mathcal{N}_0$ be fixed. Defining $g_n : \N \rightarrow \C$ by $g_n(m) = f(n + m)$ for $|m|\le 16TD$ and $0$ otherwise, we have from the definition of $\mathcal{N}_0$ that 
\[\big|\mb{E}_{h\in [TD]-[TD]}\mb{E}_{u,u'\in [T]-[T]}\mb{E}_{s,s'\in S} g_n(h+ us)\overline{g_n(h + u's')}\big|\ge \delta^2/16.\]
Note here that $|h + us|, |h + u' s'| \le 16TD$, using here that $S \subset [-4D, 4D]$.
Taking the Fourier expansion $g_n(m) = \int_{\R/\Z} \wh{g_n}(\theta) e(\theta m) d\theta$ and applying the triangle inequality, this gives
\begin{equation}\label{basic-lower}\int_{(\R/\Z)^2} \big|\wh{g}_n(\theta)\wh{g}_n(\theta')\big|  K(\theta, \theta') d\theta d\theta' \ge \delta^{2}/16,\end{equation}
where
\[ K(\theta, \theta') := \big|\mb{E}_{h\in [TD]-[TD]} e\big((\theta-\theta') h\big) \psi(\theta)\psi(\theta')\big|  \] with \begin{equation}\label{eq2point1c} \psi(\theta) :=  \mb{E}_{u\in [T]-[T]}\mb{E}_{s \in S} e (\theta us).\end{equation}
Now by bounding the $\psi(\cdot )$ terms trivially by $1$ and using that \[ |\mb{E}_{h\in [TD]-[TD]}e((\theta - \theta') h)| = |\mb{E}_{h\in [TD]}e((\theta - \theta') h)|^2 \ll (TD)^{-2}\snorm{\theta - \theta'}_{\R/\Z}^{-2},\] we have $K(\theta, \theta') \ll \min (1,(TD)^{-2}\snorm{\theta - \theta'}_{\R/\Z}^{-2} )$.
From this, Cauchy--Schwarz and Parseval it follows that 
\[ \int_{\R/\Z} \big|\wh{g_n}(\theta) \wh{g_n}(\theta + \alpha)\big| K(\theta, \theta + \alpha) d\theta \ll \Big(\int_{\R/\Z} |\wh{g}_n(\theta)|^2 \Big) (TD)^{-2} \Vert \alpha \Vert_{\R/\Z}^{-2} \ll (TD)^{-1}  \Vert \alpha \Vert_{\R/\Z}^{-2}.\] Integrating over $\alpha \in \R/\Z$, we see that the contribution to \cref{basic-lower} from $\Vert \alpha \Vert_{\R/\Z} \ge C \delta^{-2}/TD$ is negligible for $C$ sufficiently large, that is to say
\[ \int_{\Vert \theta - \theta' \Vert_{\R/\Z} \le C\delta^{-2}/TD} \big|\wh{g}_n(\theta)\wh{g}_n(\theta')\big| K(\theta, \theta')  d\theta d\theta' \ge \delta^{2}/32.\] Therefore, bounding the geometric series part of $K$ trivially by $1$,
\[ \int_{\Vert \theta - \theta' \Vert_{\R/\Z} \le C\delta^{-2}/TD} \big|\wh{g}_n(\theta)\wh{g}_n(\theta')\psi(\theta)\psi(\theta')\big| d\theta d\theta' \ge \delta^{2}/32.\] In particular, for some $\alpha \in \R/\Z$ we have
\[ \int_{\R/\Z} \big|\wh{g}_n(\theta) \wh{g}_n(\theta + \alpha)\psi(\theta)\psi(\theta+ \alpha)\big| d\theta  \gg \delta^{4}TD.\]

Using the AM--GM inequality $x^2 + y^2 \ge 2xy$ with $x = |\wh{g_n}(\theta) \psi(\theta)|$ and $y = |\wh{g_n}(\theta+\alpha)\psi(\theta+\alpha)|$, it follows that 
\begin{equation}\label{g-s-bd} \int_{\R/\Z} |\wh{g}_n(\theta)|^2 |\psi(\theta)|^2 d\theta  \gg \delta^{4}TD.\end{equation}
By Parseval's inequality we have $\int_{\R/\Z} |\wh{g}_n(\theta)|^2 \ll TD$, and so for sufficiently small $c_1$ we have
\begin{equation}\label{g-s-bd-2} \int_{|\psi(\theta)| \ge c_1 \delta^2} |\wh{g}_n(\theta)|^2 |\psi(\theta)|^2 d\theta  \gg \delta^{4}TD.\end{equation}

To proceed further we need to analyse the $\theta$ for which $|\psi(\theta)| \ge c_1 \delta^2$. Suppose in the following discussion that $\theta$ has this property. Recalling that the definition of $\psi$ is \cref{eq2point1c}, it follows that $\mb{E}_{u\in [T]-[T]}\big|\mb{E}_{s\in S}e( \theta u s)\big|\ge c_1\delta^{2}$. Writing $\mathcal{U} := \{ u \in [T] : \big|\mb{E}_{s\in S}e( \theta u s)\big|\ge c_1\delta^{2}/2\}$, we see that $\mu_{[T] - [T]}(\mathcal{U} \cup - \mathcal{U}) \gg \delta^2$, where $\mu_{[T] - [T]}$ denotes the natural weighted probability measure on $[T] - [T]$. Since $\mu_{[T] - [T]}(x) \le 1/T$ pointwise, it follows that $|\mathcal{U}| \gg \delta^2 T$. 

We now apply the diophantine assumption on $S$. We conclude that for each $u \in \mathcal{U}$ there is some nonzero $q_u \ll (L'/\delta)^{2L}$ such that $\Vert q_u u \theta \Vert_{\R/\Z} \ll (L'/\delta)^{2L}/D$. By further refining the set of $u$ (to a set of size $\gg (\delta/L')^{4L}T$) we may assume that $q_u$ does not depend on $u$. Denote this common value by $q_0$.

Now we apply \cref{vino-lemma}, taking $\alpha = q_0 \theta$, $\delta_2 \gg (\delta/L')^{4L}$ and $\delta_1 = (L'/\delta)^{2L}/D$. One can check that the conditions of \cref{vino-lemma} are consequences of the hypothesised lower bounds on $D$ and $T$, provided $C$ is large enough. The conclusion of the lemma is then that there is some $q \ll (L'/\delta)^{4L}$ such that $\Vert \alpha q\Vert_{\R/\Z} \ll (L'/\delta)^{6L}/TD$. Taking $q' := q q_0$, we see that $q' \ll (L'/\delta)^{6L}$ and $\Vert \theta q'\Vert_{\R/\Z} \ll (L'/\delta)^{8L}/TD$.

It follows from this analysis and \cref{g-s-bd} that $\int_{\theta \in \Theta } |\wh{g}_n(\theta)|^2 \gg \delta^4 TD$, where $\Theta$ is the set of all $\theta$ for which $\Vert \theta q\Vert_{\R/\Z} \le (L'/\delta)^{8L}/TD$ for some $q \in \N$ with $q \ll (L'/\delta)^{6L}$. Since the measure of $\Theta$ is $\ll (L'/\delta)^{14L}/TD$, there is some $\theta_n \in \Theta$ such that \begin{equation}\label{2.5b} |\wh{g}_n(\theta_n)| \gg (\delta/L')^{18L} TD.\end{equation} By refining $\mathcal{N}_0$ we may, using \cref{2point1-star}, find $\mathcal{N}_1 \subset \mathcal{N}_0$ such that 
\begin{equation}\label{n1-lower} \mb{E}_{n \in [N]}^{\log} 1_{\mathcal{N}_1}(n) \gg (\delta/L')^{8L}\end{equation} and such that, for all $n \in \mathcal{N}_1$, the corresponding $\theta_n$ all have the same value of $q$; that is, $\Vert q \theta_n\Vert_{\R/\Z} \ll (L'/\delta)^{8L}/TD$ for all $n \in \mathcal{N}_1$. Writing out the definition of the Fourier transform, we have from \cref{2.5b} that
\[ \Big|\mb{E}_{|m| \le 16TD} g_n(m) e(\theta_n m)\Big| \gg (\delta/L')^{18L}.\]
Recall that $H \in \N$ is a given parameter, satisfying $H \le (L'/\delta)^{50L}$. By the properties of $\theta_n$, we have
\[ \Big|\mb{E}_{h\in [H]}\mb{E}_{|m| \le 16TD} g_n(m) e(\theta_n (m - qh))\Big| \gg (\delta/L')^{18L}.\]
Substituting $m' := m - qh$ gives
\[ \Big|\mb{E}_{h\in [H]}\mb{E}_{-16TD + qh \le m' \le 16TD + qh} g_n(m' + qh) e(\theta_n m')\Big| \gg (\delta/L')^{18L},\] which implies that 
\[ \Big|\mb{E}_{h\in [H]}\mb{E}_{|m'| \le 16TD} g_n(m' + qh) e(\theta_n m')\Big| \gg (\delta/L')^{18L}\] by the bound on $H$ and \cref{ord-avg-shift}. Dropping the dashes on $m'$ and swapping the order of the averages gives
\[ \Big|\mb{E}_{|m| \le 16TD} e(\theta_n m)\mb{E}_{h\in [H]}g_n(m + qh) \Big| \gg (\delta/L')^{18L}.\] 
By Cauchy--Schwarz, it follows that 
\[ \mb{E}_{|m| \le 16TD} \mb{E}_{h, h' \in [H]} g_n(m + qh) \overline{g_n(m + qh')} \gg(\delta/L')^{36L}.\] Recall that we have this for all $n \in \mathcal{N}_1$. However, the quantity on the left is non-negative for all $n$. Taking the logarithmic average over $n$ (and recalling \cref{n1-lower}) we obtain
\[ \mb{E}_{n \in [N]}^{\log} \mb{E}_{|m| \le 16TD} \mb{E}_{h, h' \in [H]} g_n(m + qh) \overline{g_n(m + qh')} \gg(\delta/L')^{44L}.\]
Recalling that $g_n(m) = f(n+m)$, and taking the $n$ average to the inside, this is
\[ \mb{E}_{|m| \le 16TD} \mb{E}_{h, h' \in [H]} \mb{E}_{n \in [N]}^{\log}  f(n + m + qh) \overline{f(n + m + qh')} \gg(\delta/L')^{44L}.\]
Applying \cref{log-avg-shift} to the inner average for each $m$ (and using the assumed bound on $\frac{\log TD}{\log N}$) we may drop the $m$-average, obtaining
\[  \mb{E}_{h, h' \in [H]} \mb{E}_{n \in [N]}^{\log}  f(n + qh) \overline{f(n  + qh')} \gg(\delta/L')^{44L}.\]
This is equivalent to the stated result.
\end{proof}

\begin{lemma}\label{lem:input-concat-2}
There is an absolute constant $\delta_0$ such that the following holds. Fix $\delta \in (0,\delta_0]$ and $L,L',D\ge 1$. Let $S \subset [-4D, 4D]$ be a set which is $(L,L',D)$-diophantine, and let $T \in \N$ be a parameter. Let $X$ be a further sufficiently large parameter. Suppose that $D, T \ge (L'/\delta)^{8L}$ and that $TD \le (\delta/L')^{50L} X$.  Let $H$ be a positive integer with $H \le (\delta/L')^{50L} TD$. Let $f : \N \rightarrow \C$ be $1$-bounded and suppose that we have
\begin{equation}\label{lem25-assump} \mb{E}_{n \in [X]}\mb{E}_{t,t' \in [T]}\mb{E}_{s\in S}  f(n + ts) \overline{f(n + t's)} \ge \delta.\end{equation}
Then there exists $q \in \N$, $q\le (L'/\delta)^{8L}$, such that $\Vert f \Vert_{U^1[X; q,H]} \ge (\delta/L')^{25L}$. 
\end{lemma}
\begin{proof}
The same proof works essentially verbatim, except that the three applications of \cref{log-avg-shift} are replaced by appeals to \cref{ord-avg-shift}, using each time the assumption $\frac{TD}{X} \le (\delta/L')^{50L}$ rather than a bound on $\frac{\log TD}{\log X}$ in the logarithmic case. 
\end{proof}

Rather than \cref{lem:input-concat-2} itself, we will need the following iterated variant. Here we use the notation for difference operators $\Delta_{(h,h')}$ described in \cref{notation-sec}.

\begin{lemma}\label{lem:input-concat-2-iter} There are absolute constants $\delta_0 < 1$ and $C = C_{\operatorname{\ref{lem:input-concat-2-iter}}}$ such that the following holds.
Fix $\delta \in (0,\delta_0]$ and $L,L',D_1, D_2\ge 1$. For $i = 1,2$ suppose that $S_i \subset [-4D_i, 4D_i]$ is a set which is $(L, L',D_i)$-diophantine, and let $T_i$ be a parameter. Let $X$ be a sufficiently large parameter and suppose that $D_i, T_i \ge (L'/\delta)^{CL^2}$ and that $T_i D_i \le (L'/\delta)^{CL^2}X$. Let $H_1, H_2$ be positive integers with $H_i \le (L'/\delta)^{CL^2} T_i D_i$.  Let $\psi : \N \rightarrow \C$ be $1$-bounded and suppose that 
\begin{equation}\label{lem26-assump-2} \mb{E}_{n \in [X, 2X),t_1,t_1' \in [T_1], t_2, t'_2 \in [T_2],s_1\in S_1,s_2 \in S_2}  \Delta_{s_1(t_1, t'_1)}\Delta_{s_2(t_2, t'_2)}\psi(n)\ge \delta.\end{equation}
Then there exist $q_1,q_2 \in \N$, $q_i\le (L'/\delta)^{CL^2}$, such that 
\[ \mb{E}_{n \in [2X], h_1, h'_1 \in [H_1], h_2, h'_2 \in [H_2]} \Delta_{q_1(h_1, h'_1)} \Delta_{q_2(h_2, h'_2)} \psi(n) \ge (\delta/L')^{CL^2}.\]
\end{lemma}
\begin{remark} 
We have only stated a version with two difference operators (which will involve one iteration of \cref{lem:input-concat-2}), since this is what we will need later. A similar argument gives a version with $k$ difference operators.
\end{remark}
\begin{proof}
By an averaging argument, there are at least $\delta |S_2| T_2^2/2$ triples $(s_2, t_2, t'_2)$ such that 
\[ \mb{E}_{n \in [X,2X), t_1, t'_1 \in [T_1], s_1 \in S_1} \Delta_{s_1(t_1, t'_1)} \big(\Delta_{s_2 (t_2, t'_2)} \psi \big)(n) \ge \delta/2.\] 
Since, for any $n, s_1$, the average over $t_1, t'_1$ is non-negative, we have
\[ \mb{E}_{n \in [2X], t_1, t'_1 \in [T_1], s_1 \in S_1} \Delta_{s_1(t_1, t'_1)} \big(\Delta_{s_2 (t_2, t'_2)} \psi \big)(n) \ge \delta/4.\] 
For each such triple, this is exactly the hypothesis \cref{lem25-assump} of \cref{lem:input-concat-2} with $f = \Delta_{s_2(t_2, t'_2)}\psi$ (and $\delta$ replaced by $\delta/4$ and $X$ by $2X$). The conclusion of \cref{lem:input-concat-2} is then that there exists $q = q(s_2, t_2, t'_2) \le (L'/\delta)^{O(L)}$ such that $\Vert \Delta_{s_2(t_2, t'_2)}\psi \Vert_{U^1[2X; q, H_1]} \ge (\delta/L')^{O(L)}$. Squaring and writing out, this gives
\begin{equation}\label{eq322}  \mb{E}_{n \in [2X]} \mb{E}_{h_1,h'_1 \in [H_1]}\Delta_{q(h_1,h'_1)} \big(\Delta_{s_2(t_2, t'_2)} \psi\big) (n) \ge (\delta/L')^{O(L)}. \end{equation} By pigeonhole, we may pass to set of $(\delta/L')^{O(L)}|S_2| T_2^2$ triples $(s_2, t_2, t'_2)$ such that $q_1 = q(s_2, t_2, t'_2)$ is independent of $s_2, t_2, t'_2$.
Since the expression on the left in \cref{eq322} is always nonnegative, we may average over \emph{all} $(s_2, t_2, t'_2) \in S_2 \times [T_2] \times [T_2]$, obtaining
\[\mb{E}_{h_1, h'_1 \in [H_1]} \mb{E}_{n \in [2X], t_2, t'_2 \in T_2, s_2 \in S_2} \Delta_{s_2(t_2, t'_2)} \big(\Delta_{q_1(h_1,h'_1)} \psi\big)(n) \ge (\delta/L')^{O(L)}.\]
For at least $(\delta/L')^{O(L)} H_1^2$ pairs $(h_1, h'_1)$, we have 
\[ \mb{E}_{n \in [2X], t_2, t'_2 \in T_2, s_2 \in S_2} \Delta_{s_2(t_2, t'_2)} \big(\Delta_{q_1(h_1,h'_1)} \psi\big)(n) \ge (\delta/L')^{O(L)}.\]
For each such pair, this is again the hypothesis \cref{lem25-assump} of \cref{lem:input-concat-2}, now with $f = \Delta_{q_1(h_1, h'_1)}\psi$, $\delta$ replaced by $(\delta/L')^{O(L)}$, and again with $N = 2X$. Another application of \cref{lem:input-concat-2} gives that there exists $q = q(h_1, h'_1) \le (L'/\delta)^{O(L^2)}$ such that $\Vert \Delta_{q_1(h_1, h'_1)}\psi\Vert_{U^1[2X; q,H_2]} \ge (\delta/L')^{O(L^2)}$, provided that $C$ is sufficiently large that the relevant conditions on $D_2, T_2$ and $T_2D_2/X$ are satisfied. Squaring and writing out, this gives
\begin{equation}\label{eq323} \mb{E}_{n \in [2X]} \mb{E}_{h_2, h'_2 \in [H_2]} \Delta_{q(h_2, h'_2)} \Delta_{q_1(h_1, h'_1)} \psi(n) \ge (\delta/L')^{O(L^2)}.\end{equation} Passing to a further subset of $(\delta/L')^{O(L^2)}$ pairs $(h_1, h'_1)$, we may assume that $q_2 = q(h_1, h'_1)$ does not depend on $(h_1, h'_1)$. Since the expression on the left in \cref{eq323} is non-negative for all $q$, we obtain the desired result by averaging over $h_1, h'_1$.
\end{proof}

\section{Diophantine properties of almost primes}\label{sec3}

The main result of this section, \cref{lem:crit-estimate}, is a vital technical ingredient in our later arguments. Roughly, it states that sets such as $\{ p_1 \cdots p_k : p_i \in \mc{P}_i\}$ and $\{ p^2_1 \cdots p^2_k : p_i \in \mc{P}_i\}$ are diophantine (see \cref{dioph-def}) with suitable parameters, where $\mc{P}_i$ are dyadically localised sets of primes.  We first note a general lemma for `bilinear' exponential sums. 

\begin{lemma}\label{bilinear-standard}
Let $j \in \N$. Let $\delta \in (0,\frac{1}{2})$, and let $S_1 \subset [N_1]$ and $S_2 \subset [N_2]$ be sets with $|S_i| = \sigma_i N_i$ for $i = 1,2$.  Suppose that, for some $\theta \in \R/\Z$, we have $|\mb{E}_{s_1 \in S_1, s_2 \in S_2} e(\theta s^j_1 s^j_2)| \ge \delta$. Then either $N_i \le (\sigma_1 \sigma_2 \delta)^{-O_j(1)}$ for some $i \in \{1,2\}$, or else there is some $q \in \N$, $q \le (\delta \sigma_1\sigma_2)^{-O_j(1)}$, such that $\Vert q \theta\Vert_{\R/\Z} \le (\delta \sigma_1\sigma_2)^{-O_j(1)}(N_1 N_2)^{-j}$.
\end{lemma}
\begin{remark} We will only need the cases $j = 1,2$, in which case of course the exponents may be taken to be absolute constants. \end{remark}
\begin{proof}
Write the condition as
\[ \big| \mb{E}_{n_1 \in [N_1], n_2 \in [N_2]} 1_{S_1}(n_1) 1_{S_2}(n_2) e(\theta n^j_1 n^j_2) \big| \ge \delta \sigma_1 \sigma_2 .\]
By two applications of the Cauchy--Schwarz inequality, we obtain
\[ \mb{E}_{n_1, n'_1 \in [N_1], n_2, n'_2 \in [N_2]} e \big( \theta (n^j_1 - n^{\prime j}_1)(n^j_2 - n^{\prime j}_2)\big) \ge (\delta \sigma_1\sigma_2)^{4} .\]
To handle this, we use the `log-free' multidimensional Weyl inequality \cite[Proposition~2.2]{GT14}; we remark that the published version of that paper omits the necessary constraint that $\min(N_i)$ be sufficiently large. 
\end{proof}

We now proceed to the main technical lemma of the section. Although we will only need this lemma for $j = 1,2$, it is no harder to prove it for general $j$.
\begin{lemma}\label{lem:crit-estimate}
Let $j \in \N$. Then there is a constant $L_j \ge 1$ such that the following holds. Let $k \ge 2$ be a natural number and let $\delta \in (0,\frac{1}{2})$. Let $M_1,\ldots, M_k$ be a sequence of integers such that the intervals $[M_i, (1 + \frac{1}{4k}) M_i)$ are disjoint. Set $Q := k^k \prod_{i = 1}^k \log  M_i$, and suppose the condition $\min_i M_i > Q^{L_j}$ is satisfied. For each $i$, suppose we are given a parameter $\eta_i$ satisfying $\frac{1}{8k} \le \eta_i \le \frac{1}{4k}$ and define $\mc{P}_i$ to be the set of primes satisfying $M_i \le p < M_i(1 +\eta_i)$, and set $S := \{p_1^j \cdots p_k^j : p_i \in \mc{P}_i\}$. Then $S$ is $(L_j, k, (M_1 \cdots, M_k)^j)$-diophantine.
\end{lemma}

\begin{proof} 
We first note that the case $k = 1$ is also true and is essentially a standard result about exponential sums over powers of primes. We in fact need (a slight generalisation of) this result in our proof. Since it is hard to find an appropriate reference with the log-free bound that we require, we give this in \cref{log-free-weyl-2}. 

Suppose now that $k \ge 2$. Without loss of generality, assume $M_1 > M_2 > \cdots > M_k \ge 3$. Let $\theta \in \R/\Z$ and suppose that
\begin{equation}\label{lem32-assump} \big|\mb{E}_{p_1 \in \mc{P}_1,\dots, p_k \in \mc{P}_k}e(\theta p_1^j\cdots p_k^j)\big|\ge \delta.\end{equation}
We must show that there is some $q \in \N$ such that 
\begin{equation}\label{q-desire-conclusion} q \le (k/\delta)^{L_j} \quad \mbox{and} \quad \Vert q \theta \Vert_{\R/\Z} \le (k/\delta)^{L_j} (M_1\cdots M_k)^{-j}.\end{equation} 
We try applying \cref{bilinear-standard} with $N_1 := 2 \prod_{i \le k : i \operatorname{even}} M_i$ and $N_2 := 2 \prod_{i \le k : i \operatorname{odd}} M_i$. Define sets $S_1 \subset [N_1]$, $S_2 \subset [N_2]$ by $S_1 := \prod_{i \le k: i \operatorname{even}} \mc{P}_i$ and $S_2 := \prod_{i \le k: i \operatorname{odd}} \mc{P}_i$ (the stated containments are easily verified). Set $\sigma_i := |S_i|/N_i$. Since $M_i > Q > k^k$, it follows from the prime number theorem with classical error term (see e.g. \cite[Section~5.6]{IK-book}) that we have $|\mc{P}_i| \ge cM_i/ k \log M_i$ for some absolute $c > 0$. Therefore
\[ \sigma_1\sigma_2 = \frac{1}{4} \prod_{i=1}^{k}\frac{|\mc{P}_i|}{M_i}\ge \Big(\frac{c}{2k}\Big)^{k} \prod_{j=1}^{k} \frac{1}{\log M_i} \ge \big(\frac{c}{2}\big)^k Q^{-1} \gg Q^{-2},\] using in this last step that $Q > k^k$. 
Applying \cref{bilinear-standard}, and noting that $N_1 \le N_2$, it follows that either
\begin{equation}
\label{small-n-case-1}
N_1 \le (Q/\delta)^{O_j(1)}
\end{equation}
or else there is some $q \in \N$ with
\begin{equation}\label{q-exist} q \le  (Q/\delta)^{O_j(1)} \quad \mbox{and} \quad \Vert \theta q \Vert_{\R/\Z} \le (Q/\delta)^{O_j(1)} (M_1 \cdots M_k)^{-j}.\end{equation}
We leave aside \cref{small-n-case-1} for now, and assume that \cref{q-exist} holds.
If $\delta \le 1/Q$ then \cref{q-desire-conclusion} follows immediately (with $L_j$ equal to twice the $O_j(1)$ exponent). Therefore we may suppose henceforth that $\delta \ge 1/Q$. In particular, \cref{q-exist} gives (after doubling the implied constant in the exponents) that 
\[ q \le  Q^{O_j(1)} \quad \mbox{and} \quad \Vert \theta q \Vert_{\R/\Z} \le Q^{O_j(1)} (M_1 \cdots M_k)^{-j}.\]
Thus $\theta = \frac{a}{q} + \theta'$ for some $a \in \Z$, with 
\begin{equation}\label{theta-dash-upper} |\theta'| \le Q^{O_j(1)} (M_1 \cdots M_k)^{-j}.\end{equation}
We now return to the original sum \cref{lem32-assump}. By pigeonhole, there is a choice of $t = p_2^j \cdots p_k^j$ such that $\big| \mb{E}_{p_1 \in \mc{P}_1} e(\theta t p_1^j) \big| \ge \delta$. By \cref{log-free-weyl-2} (that is, essentially the case $k = 1$ of the present lemma) it follows that there is some $q_0 \le (k/\delta)^{O_j(1)}$ such that $\Vert \theta t q_0 \Vert_{\R/\Z} \le (k/\delta)^{O_j(1)} M_1^{-j}$. Since $\theta' = \theta - a/q$, this means that $\theta' t q_0$ is within $(k/\delta)^{O_j(1)} M_1^{-j}$ of $-atq_0/q$, an integer multiple of $1/q$. However, we may also note using \cref{theta-dash-upper} and the bound $p_2 \cdots p_k \le 3 M_2 \cdots M_k$ that 
\[ |\theta' t q_0| \le Q^{O_j(1)} (M_1 \cdots M_k)^{-j} \cdot (3 M_2\cdots M_k)^j \cdot (k/\delta)^{O_j(1)} \le Q^{O_j(1)} M_1^{-j} < \frac{1}{2q}. \] Here, in the penultimate step we used that $\delta \le 1/Q$ (and so $k/\delta \le Q^2$), and in the last step we invoked the assumption $M_1 > Q^{L_j}$ and the upper bound $q \le Q^{O_j(1)}$ (and assumed $L_j$ is large enough). Since $(k/\delta)^{O_j(1)} M_1^{-j} \le Q^{O_j(1)} M_1^{-j} < \frac{1}{2q}$ (for the aforementioned reasons) the only possible integer multiple of $1/q$ that $\theta' t q$ can be near is $0$, and therefore $|\theta' t q_0| \le (k/\delta)^{O_j(1)} M_1^{-j}$ and $q \mid at q_0$. Dividing through by $tq_0$, we obtain $|\theta'| \le (k/\delta)^{O_j(1)} (M_1 \cdots M_k)^{-O_j(1)}$. Note also that $(q,t) = 1$ since all prime factors of $t$ are at least $M_k > Q^{L_j} \ge q$, and therefore $q \mid aq_0$. Finally it follows that $\Vert \theta q_0 \Vert_{\R/\Z} = \Vert \theta' q_0 \Vert_{\R/\Z} \le |\theta'|q_0 \le (k/\delta)^{O_j(1)} (M_1 \cdots M_k)^{-j}$, which is the desired conclusion \cref{q-desire-conclusion}.

It remains to analyse the `small parameter' case \cref{small-n-case-1}, that is to say $N_1 \le (Q/\delta)^{O_j(1)}$. The assumption $\min_i M_i > Q^{L_j}$ certainly implies that $N_1 > Q^{L_j}$. Therefore (assuming $L_j$ large enough) we have $N_1 \le \delta^{-O_j(1)}$. It follows that \begin{equation}\label{m2mkbd} M_2 \cdots M_k = \prod_{\substack{i \le k \\ i \operatorname{even}}} M_i \cdot \prod_{\substack{i \le k - 1 \\ i \operatorname{even}}} M_{i+1} \le \prod_{\substack{i \le k \\ i \operatorname{even}}} M_i  \cdot \prod_{\substack{i \le k - 1 \\ i \operatorname{even}}} M_{i} \le N_1^2 \le \delta^{-O_j(1)}.\end{equation} As before, \cref{lem32-assump} implies that there is some $t = p_2^j \cdots p_k^{j}$ such that $|\mb{E}_{p_1 \in \mc{P}_1} e(\theta t p_1^j)| \ge \delta$. By \cref{log-free-weyl-2} it follows that there is some $q_0 \le (k/\delta)^{O_j(1)}$ such that $\Vert \theta t q_0 \Vert_{\R/\Z} \le (k/\delta)^{O_j(1)} M_1^{-j}$. Taking $q := tq_0$, we then have (using \cref{m2mkbd})
\[ q \le (M_2 \cdots M_k)^j (k/\delta)^{O_j(1)} \le (k/\delta)^{O_j(1)},\] and (using \cref{m2mkbd} again) $\Vert \theta q \Vert_{\R/\Z} \le (k/\delta)^{O_j(1)}M_1^{-j} \le (k/\delta)^{O_j(1)} (M_1 \cdots M_k)^{-j}$, which is once again the desired conclusion \cref{q-desire-conclusion}.
\end{proof}

\section{Fourier decomposition of a majorant for the primes}\label{sec-fourier-decomp}

In this section we give another technical ingredient for our later arguments. Here is the main result.

\newcommand\per{\operatorname{per}}
\begin{lemma}\label{lem4point3}
Let $X$ be a large parameter. Then there is a function $\tilde\Lambda : [X, 2X) \rightarrow \R_{\ge 0}$ with 
\begin{equation} \label{majorant} \tilde\Lambda(p) \gg \log X\end{equation} for all primes $p \in [X, 2X)$ and 
\begin{equation}\label{tilde-lam} \mb{E}_{x \in [X, 2X)} \tilde\Lambda(x) \ll 1\end{equation} such that the following is true. Let $c \in (0,1)$ be a constant. For any parameter $Q  \in \N$, $Q \le \log X$, there is a $(Q!)$-periodic function $\Lambda_{\per}$ satisfying
\begin{equation}\label{lam-per-props} \mb{E}_{x \in [X, 2X)} |\Lambda_{\per}(x)| \ll (\log Q)^{O(1)}\quad \mbox{and} \quad \Vert \Lambda_{\per} \Vert_{\infty} \ll Q^{2} \end{equation} together with a decomposition $\tilde\Lambda - \Lambda_{\per} = \sum_{i} g_i + h$ \textup{(}where the sum over $i$ is finite\textup{)} with the following properties. First, the function $h$ is small in $\ell^1$ in the sense that 
\begin{equation}\label{small-h-ell1} \mb{E}_{x \in [X, 2X)} |h(x)|  \ll Q^{-1}.\end{equation} Second, the functions $g_i$ are reasonably bounded in sup norm in the sense that
\begin{equation}\label{gi-ell-infty} \sum_i \Vert g_i \Vert_{\infty} \ll (\log X)^{O_c(1)}\end{equation} for all $i$. Finally, denoting $\Vert \wh{f} \Vert_{\infty} := \sup_{\theta \in \R/\Z} \big|\sum_{x \in [X,2X]} f(x) e(\theta x)\big|$ we have the estimate
\begin{equation}\label{sum-gi-gi-hat} \sum_i \Vert \wh{g}_i \Vert_{\infty}^c \Vert g_i \Vert_{\infty}^{1 - c} \ll X^{c} Q^{-c/4}.\end{equation}
Here, all implied constants may depend on $c$ but are effectively computable.
\end{lemma}
\begin{proof} We take $\tilde\Lambda$ to be a Selberg-type majorant for the primes. Rather than describe the construction explicitly here, we can just refer to \cite[Proposition 3.1]{green-tao-selberg}, which provides the relevant properties. Taking $F(n) = n$ in that proposition (thus the singular series $\mathfrak{S}_F$ as defined in \cite{green-tao-selberg} is $\asymp 1$) and $R := X^{1/10}$, we can take $\tilde\Lambda = \beta$, where $\beta$ is the function constructed in \cite[Proposition 3.1]{green-tao-selberg}. The desired majorant property \cref{majorant} is a consequence of \cite[Equation (3.1)]{green-tao-selberg}. 
The bound \cref{tilde-lam} is an absolutely standard fact about the Selberg sieve. It could be deduced within the framework of \cite{green-tao-selberg} by summing \cite[Equation (3.3)]{green-tao-selberg} over $n \in [X, 2X)$, and discarding the negligible contribution from all frequencies except $a/q = 0$. 
On the Fourier side we have (see \cite[Equation (7.7)]{green-tao-selberg})
\[ \tilde{\Lambda}(n) = \Big( \sum_{q \le R} \frac{\mu(q)}{\phi(q)}\Big)^{-1}\Big( \sum_{q \le R} \frac{\mu(q)}{\phi(q)} \sum_{(a,q) = 1} e\big(\frac{an}{q}\big) \Big)^2. \]
It is shown in \cite[Proposition 7.1]{green-tao-selberg}, following Ramar\'e and Ruzsa \cite{ramare-ruzsa}, that 
\[ \tilde\Lambda(n) = \sum_{q \le R^2} c_q \sum_{(a,q) = 1} e(an/q)\]
with $c_q$ supported on squarefrees with $q \le R^2$ and $|c_q| \ll \tau(q)^2/q$. Set $i_0 := \lfloor \log_2 Q\rfloor$, $i_1 :=  \lfloor A \log_2 \log X\rfloor$ for some $A = A(c)$ to be specified below, and finally set
\[ \Lambda_{\per}(n) := \sum_{q \le 2^{i_0}} c_q \sum_{(a,q) = 1} e(an/q), \qquad f_i(n) := \sum_{2^i <  q \le 2^{i+1}} c_q \sum_{(a,q) = 1} e(an/q)\] for $i_0 \le i < i_1$ and 
\[ f_{i_1}(n) := \sum_{2^{i_1} \le q \le R^2} c_q \sum_{(a,q) = 1} e(an/q).\]
It is then clear that $\Lambda_{\per}$ is $(Q!)$-periodic and that $\tilde\Lambda - \Lambda_{\per} = \sum_i f_i$. We now define $g_i, g'_i$ by `thresholding' the $f_i$, specifically by setting
\[ g_i(n) := f_i(n) 1_{|f_i(n)| \le 2^{ic/2}}, \qquad g_i'(n) := f_i(n) 1_{|f_i(n)| > 2^{ic/2}}\] for $i_0 \le i \le i_1$. Set $h := \sum_i g'_i$. The $\ell^{\infty}$ bound \cref{gi-ell-infty} is then immediate.

Next we establish \cref{small-h-ell1}. For this, we will use the moment estimates
\begin{equation}\label{ell-m-est} \mb{E}_{x \in [X, 2X)} |f_i(x)|^m  \ll_m i^{C_m}, \quad  i \in [i_0, i_1),\quad \mbox{and} \quad \mb{E}_{x \in [X, 2X)} |f_{i_1}(x)|^m  \ll_m (\log X)^{C_m}\end{equation} for $m \in \N$ and for some constants $C_m$, which we will establish below. Indeed, taking $m = \lceil 4/c\rceil$ in \cref{ell-m-est} yields
\begin{equation}\label{thresholding-1} \mb{E}_{x \in [X, 2X)} |g'_i(x)| \le 2^{ci ( 1 - m)/2} \mb{E}_{x \in [X, 2X]} |f_i(x)|^m \ll 2^{ci (1 - m)}i^{C_m} \ll 2^{-i}, \end{equation} uniformly for $i \in [i_0, i_1)$,
and similarly
\begin{equation}\label{thresholding-2} \mb{E}_{x \in [X, 2X)} |g'_{i_1}(x)| \le (\log X)^{cA(1 - m)/2} \mb{E}_{x \in [X, 2X]} |f_{i_1}(x)|^m \ll (\log X)^{cA(1 - m) + C_m} \ll (\log X)^{-A} \end{equation} provided $A$ is chosen large enough (depending only on $c$). The desired estimate \cref{small-h-ell1} is now immediate from the triangle inequality, the dominant contribution being from \cref{thresholding-1} with values $i \approx i_0$. (Here we use the assumption that $Q \le \log X$ to guarantee that the contribution from \cref{thresholding-2} is insignificant.)

Now we establish \cref{sum-gi-gi-hat}. It is enough to show that
\begin{equation}\label{gi-g-ell-infty} \Vert \wh{g_i} \Vert_{\infty} \ll X2^{-3i/4}\end{equation} for $i \in [i_0, i_1]$, since the desired estimate then follows using the $\ell^{\infty}$ bounds on the $g_i$ implicit in the definitions of these functions.
To show \cref{gi-g-ell-infty}, it suffices to show the non-thresholded estimates
\begin{equation}\label{fi-f-ell-infty} \Vert \wh{f_i} \Vert_{\infty} \ll X2^{-3i/4}\end{equation} for $i \in [i_0, i_1]$, from which \cref{gi-g-ell-infty} follows using \cref{thresholding-1,thresholding-2}.
From the definition of $f_i$, summing the geometric series and the bound $|c_q| \ll \tau(q)^2/q$, we have
\[ \big|\sum_{x\in [X,2X]} f_i(x) e(\theta x)\big|  \ll \sum_{2^i < q \le R^2} \frac{\tau(q)^2}{q} \sum_{(a,q) = 1} \min \big(X,  \Vert \theta - a/q \Vert_{\R/\Z}^{-1} \big) \] Since the fractions $a/q$ are $R^{-4}$-separated, the contribution from all except at most one $a/q$ will be (crudely) $\ll R^2 \cdot R^{4} \ll X 2^{-i}$. For the fraction $a/q$ closest to $\theta$, we have the trivial bound $\ll X\tau(q)^2/q$, which is $<X q^{-3/4} \ll X2^{-3i/4}$ by the divisor bound, and \cref{fi-f-ell-infty} (and therefore \cref{sum-gi-gi-hat}) follows.

We now return to establish the moment estimate \cref{ell-m-est}. An ingredient in the proof will be the (standard) estimate
\begin{equation}\label{divisor-avg} \sum_{P^+(d) \le Q} \frac{\tau(d)^C}{d} \ll_C (\log Q)^{2^C}.\end{equation}
To prove this, observe that the LHS is $\prod_{p \le Q} (1 + \frac{2^C}{p} + \frac{3^C}{p^2} + \dots) \ll_C \prod_{p \le Q} (1 + \frac{1}{p})^{2^C}$.

Turning to \cref{ell-m-est} itself, it suffices to prove the general estimate
\begin{equation}\label{gen-moment} \mb{E}_{x \in [X, 2X)} |f(x)|^m \ll (\log Q)^{O_{m,B}(1)},\end{equation} for $m \in \N$, where 
\[ f(x) = \sum_{P^+(q) \le Q} c_q \sum_{(a, q) = 1} e(\frac{an}{q}),\] the $c_q$ are supported on squarefrees and $|c_q| \le \tau(q)^{B}/q$. To prove such an estimate, we first write $f$ in physical space using Kluyver's identity $\sum_{(a,q) = 1} e(an/q) = \sum_{d \mid (n, q)} d \mu(q/d)$ for Ramanujan sums. This gives 
\begin{equation}\label{f-spatial} f(n) = \sum_{\substack{P^+(d) \le Q \\ d \mid n}} d \sum_{\substack{d \mid q \\ P^+(q) \le Q}} \mu\big(\frac{q}{d}\big) c_{q} = \sum_{\substack{P^+(d) \le Q \\ d \mid n}} \lambda_d, \quad \mbox{where} \quad \lambda_d := d\sum_{\substack{d \mid q \\ P^+(q) \le Q}} \mu\big(\frac{q}{d}\big) c_q.\end{equation} Now we have
\begin{equation}\label{eq4002} |\lambda_d| \le d \sum_{\substack{d \mid q \\ P^+(q) \le Q}} |c_q| \le \sum_{P^+(k) \le Q} \frac{\tau(kd)^B}{k} \le \tau(d)^B \sum_{P^+(k) \le Q} \frac{\tau(k)^B}{k} \ll \tau(d)^B (\log Q)^{2^B}\end{equation} by \cref{divisor-avg}.
Now observe that
\begin{align}\nonumber
\mb{E}_{n \in [X, 2X)} \big( \sum_{\substack{P^+(d) \le Q \\ d \mid n}} \tau(d)^B\big)^m & = \mb{E}_{n \in [X, 2X)}\sum_{P^+(d_1),\dots, P^+(d_m) \le Q} \big(\tau(d_1) \cdots \tau(d_m)\big)^B 1_{[d_1,\dots, d_m]\mid n}  \\ \nonumber &  \le \mb{E}_{n \in [X, 2X)} \tau([d_1,\dots, d_m])^{mB} 1_{[d_1,\dots, d_m] \mid n} \\  & \le \mb{E}_{n \in [X, 2X)} \sum_{P^+(d) \le Q} \tau(d)^{mB + m} 1_{d \mid n} \nonumber \\ & \ll \sum_{P^+(d) \le Q} \frac{\tau(d)^{mB + m}}{d} \ll (\log Q)^{O_{B,m}(1)},\label{eq4001}
\end{align} In the middle step here the key point was that the number of representations of $d$ as $[d_1,\dots, d_m]$ is at most $\tau(d)^m$, and in the penultimate step that $\mb{E}_{n \in [X, 2X)} 1_{d \mid n} \ll 1/d$ for all $d$. Combining \cref{f-spatial,eq4002,eq4001} gives \cref{gen-moment}, and so \cref{ell-m-est} follows.

The final task is to establish \cref{lam-per-props}. The first statement is immediate from \cref{gen-moment} and Cauchy--Schwarz. For the second statement (which is rather crude) one can proceed directly from the definition of $\Lambda_{\per}(n)$ using $|c_q| \ll 1$.
\end{proof}

We remark that from the first bound in \cref{lam-per-props} and the $Q!$-periodicity of $\Lambda_{\per}$ (or by direct proof) we have
\begin{equation}
\label{lam-short-interval} \mb{E}_{x \in I} \Lambda_{\per} (x) \ll (\log Q)^{O(1)} \end{equation}
for any interval of length $Q!$.

\begin{remarks}
It is possible to establish an analogue of \cref{lem4point3} with $\tilde\Lambda$ equal to the von Mangoldt function itself, taking $\Lambda_{\per}$ and the $g_i$ to be suitable Cram\'er approximants to the von Mangoldt function and $h = 0$. The details necessary to accomplish this may be found in \cite{Gre05}, though the context there was different. There are some advantages to this, for instance $\Lambda_{\per}$ is non-negative and subject to good $\ell^1$- and $\ell^{\infty}$-bounds. The drawback of proceeding this way is that the bounds are ineffective due to an application of the Siegel-Walfisz theorem. This can be corrected via the introduction of appropriate `Siegel-modified Cram\'er approximants' as in \cite{TT25} but this is quite technical. By passing to a suitable majorant as in \cref{lem4point3} we can avoid all Siegel zero issues entirely.
\end{remarks}

\section{An inverse theorem}\label{sec:inv-thm}

In this section we explore the consequences of an assumption
\begin{equation}\label{key-assump} \big| \mb{E}_{n \in [N],p \in \mc{P},p' \in \mc{P}'}^{\log}  f_1(n + \lambda p p') f_2(\lambda n p p')\big| \ge \delta \end{equation}
where $f_1, f_2 : \N \rightarrow \C$ are $1$-bounded, $\mc{P}$ consists of primes, $\mc{P}'$ of almost primes and $\lambda \in \N$ is some parameter. The reason for being interested in such an assumption was sketched in \cref{outline} and will be further apparent in \cref{sec6}.

The aim is to show that \cref{key-assump} implies that $\Vert f_1\Vert_{U^1_{\log}[N; q, H]}$ is large for suitable parameters $q, H$. (Recall from \cref{gp-local-def} the definition of these norms.) This result is directly inspired by \cite[Theorem 3.5]{Ric25}, a connection we shall elaborate upon later. Here is the technical statement of our main result.

\begin{proposition}\label{main-sec3}
There is an absolute constant $C \in \N$ such that the following holds. Let $\delta \in \R$ be a sufficiently small parameter and let $k \in \N$. Suppose that $\max( k, 1/\delta) \le \log \log N$ and $k \le \delta^{-10}$. Let $P_1, P_2, P'_1, P'_2$ be parameters with $\exp \exp ((\log \log N)^{1/10}) \le P'_1 < P'_2 < P_1 < P_2 < \exp((\log N)^{1/4})$ and $P_1 \ge (P'_2)^{10}$. Suppose that $\lambda \in \N$ satisfies $\lambda \le \exp((\log N)^{1/4})$, and that all prime factors of $\lambda$ are less than $P'_1$. Let $\mc{P}$ denote the set of primes in $[P_1, P_2)$ and suppose that $\mc{P}' \subset [P'_1, P'_2)$ is a set of `almost primes' of the following form: $\mc{P}' = \{ p_1 \cdots p_k : p_{\ell} \in I_{\ell}\}$, where $I_1,\dots, I_k \subset [P'_1, P'_2)$ are disjoint intervals, all with $\log \log (\max(I_{\ell})) - \log \log (\min(I_{\ell})) \ge k\delta^{-4.1}$, and the $p_{\ell}$ range over all primes in $I_{\ell}$ for $\ell \in [k]$.
Set $V := \lfloor \delta^{-C}\rfloor!$.  Suppose we have \cref{key-assump}. Then we have $\Vert f_1 \Vert_{U^1_{\log}[N; \lambda V, H]} \gg \delta^{O(1)}$ for any $H \in \N$ with $H \le P_1^{1/8}$.
\end{proposition}
\begin{remarks} 
For the rest of the section we write $\eps_0 := \frac{1}{10}$, thus the lower bound on $\log \log (\max(I_{\ell})) - \log \log (\min(I_{\ell}))$ is $k\delta^{-4 - \eps_0}$. Any sufficiently small absolute constant $\eps_0$ would do here. More generally, several of the assumptions on parameters are made so as to be comfortable for the required application and we do not claim these conditions are tight. For instance, the lower bound $k\delta^{-4-\eps_0}$ could be $k \delta^{-4} (\log (1/\delta))^C$ for an appropriate $C$.
\end{remarks}

\subsection{Setting up the proof of the inverse theorem}\label{subsec41}
The proof of \cref{main-sec3} is somewhat lengthy. We prepare the ground by defining some key parameters and observing simple preliminary bounds. In the proof $C_1 < C_2$ are absolute constants, with $C_1$ assumed to be sufficiently large and $C_2$ assumed sufficiently large in terms of $C_1$. We will write $Q := \lfloor \delta^{-C_2}\rfloor$.

Next we point out some consequences of the (somewhat elaborate) conditions on parameters in the statement of \cref{main-sec3}. First, the $P'_i, P_i$ are enormously larger than powers of $Q!$ (and a fortiori powers of $\delta^{-O(1)}$). Indeed $P'_1 \ge \exp\exp((\log \log N)^{1/10})$ whilst $Q! \le \exp(\delta^{-O(C_2)}) \le \exp((\log \log N)^{O(C_2)})$, using here the assumption that $1/\delta \le \log \log N$.

Second, we have \begin{equation}\label{p1p2k} P'_1 > (k \log P'_2)^{kL}\end{equation} for any fixed constant $L$ (assuming $N$ sufficiently large in terms of $L$). This is easily confirmed using the assumptions $P'_1 > \exp\exp((\log \log N)^{1/10}) > \exp((\log \log N)^3)$, $P'_2 \le N$ and $k \le \log \log N$, and will be used (twice) to verify the key condition in \cref{lem:crit-estimate}.

Third and finally, we note that all the $P_i$ parameters are significantly smaller than $N$, and one has for example $\frac{\log P_2}{\log N} \ll \delta^{10}$, which will be used several times in the analysis to assert that error terms coming from \cref{log-avg-shift} are negligible.

Next we record the fact that, under the stated conditions, the elements of $\mc{P}'$ are almost pairwise coprime. If $\mathcal{N}$ is any finite set of positive integers, we define
$\gamma(\mathcal{N}) := \mb{E}_{n, n' \in \mathcal{N}}^{\log} (n, n') - 1$, where $(n,n')$ is the gcd of $n,n'$. This is a measure of the pairwise coprimality of elements of $\mathcal{N}$; note that $\gamma(\mathcal{N}) \ge 0$ always, and that if $\gamma(\mathcal{N})$ is small then we expect the elements of $\mathcal{N}$ to be mostly coprime. Recall that $\eps_0 := \frac{1}{10}$ (though this is irrelevant to the following lemma).
\begin{lemma}\label{coprimality-lem}
Under the conditions of \cref{main-sec3}, we have $\gamma(\mc{P}') \le \delta^{4 + \eps_0/2}$.
\end{lemma}
\begin{proof}
If $\mc{P}_*$ is a set of primes and $p, p' \in \mc{P}_*$ then $(p, p') = 1$ unless $p = p'$, and so if we denote by $\mc{P}_{\ell}$ the set of primes in $I_{\ell}$ we have
 \begin{equation}\label{gam-primes} \gamma(\mc{P}_{\ell}) = \Big( \sum_{p \in \mc{P}_{\ell}} \frac{1}{p} \Big)^{-2}\sum_{p \in \mc{P}_{\ell}} \frac{p-1}{p^2} < \Big(\sum_{p \in \mc{P}_{\ell}} \frac{1}{p} \Big)^{-1}.\end{equation}
Now since $\log \log (\max(I_{\ell})) - \log \log (\min(I_{\ell})) \ge k\delta^{-4 - \eps_0}$, it follows from Mertens' theorem (see e.g. \cite[Theorem~5.4]{Kou19}) and \cref{gam-primes} that we have
$\max_{\ell} \gamma(\mc{P}_{\ell}) \le 2\delta^{4+\eps_0}/k$.
It follows that
\begin{align*} \gamma(\mc{P}') & = \mb{E}^{\log}_{p_1, p'_1 \in \mc{P}_1, \dots, p_k, p'_k \in \mc{P}_k} (p_1 \cdots p_k, p'_1,\cdots, p'_k) - 1 =  \prod_{\ell = 1}^k \mb{E}^{\log}_{p_{\ell}, p'_{\ell} \in \mc{P}_{\ell}} (p_{\ell}, p'_{\ell}) - 1 \\ & = \prod_{\ell = 1}^k (1 + \gamma(\mc{P}_{\ell})) - 1 \le \Big(1 + \frac{2\delta^{4+\eps_0}}{k}\Big)^k - 1 \le e^{2\delta^{4+\eps_0}} - 1 \le \delta^{4+\eps_0/2} .\qedhere\end{align*}
\end{proof}
As we said, the proof of \cref{main-sec3} is lengthy. Moreover, the logic is somewhat complicated, since it is difficult to state self-contained intermediate lemmas. For reference we summarise the proof structure now. 
\begin{itemize}
    \item We proceed directly from the assumption \cref{key-assump} via a series of steps to show that either \cref{eq304a} or \cref{eq304b} below holds.
    \item We then aim to show that \cref{eq304a} leads to a contradiction. This is first done subject to an unproven claim \cref{gen-avg-tau}.
    \item Claim \cref{gen-avg-tau} is proven by contradiction. This task is quickly reduced to showing that statements \cref{gen-avg-tau-2} and \cref{tau-interval} imply \cref{large-psi-fourier}, which is then a somewhat lengthy undertaking.
    \item At this point we have confirmed that \cref{eq304a} cannot hold. Therefore (by the first bullet point) \cref{eq304b} holds.
    \item We then proceed directly from \cref{eq304b} to the desired conclusion via a quite lengthy (but linear) sequence of manipulations.    
\end{itemize}

\subsection{Proof of the inverse theorem}

We turn now to the proof of \cref{main-sec3}. 

\begin{proof} Throughout the proof we will freely use the fact that $N$ is sufficiently large and that $\delta$ is sufficiently small. The starting assumption is \cref{key-assump}. We start by removing the function $f_2$ using essentially the same manipulation as in \cite[Theorem 5.2]{Ric25}. First observe that, for each $p, p'$, an application of \cref{lem:dilate} yields
\[ \mb{E}_{n \in [N]}^{\log}  f_1(n + \lambda p p') f_2(\lambda n p p') = \mb{E}_{n \in [N]}^{\log}  p' \mbf{1}_{p' \mid n} f_1\big(\frac{n}{p'}  + \lambda p p'\big) f_2(\lambda n p p') + O\big(\frac{\log P_2}{\log N}\big).\] Averaging over $p'$ (and using the upper bound $\frac{\log P_2}{\log N} \ll \delta^3$) gives
\[ \mb{E}^{\log}_{n \in [N],p' \in \mc{P}'} f_1(n + \lambda p p') f_2(\lambda n p p') = \mb{E}^{\log}_{n \in [N],p' \in \mc{P}'}  p' \mbf{1}_{p' \mid n} f_1\big(\frac{n}{p'}  + \lambda p p'\big) f_2(\lambda n p) + O(\delta^3).\] Write $g(p)$ for the expression on the left, and $\tilde g(p)$ for the first expression on the right; thus $\tilde g(p) = g(p) + \eps(p)$ with $|\eps(p)| \ll \delta^3$.
Now the assumption is that $|\mb{E}^{\log}_{p \in \mc{P}} g(p)| \ge \delta$. By Cauchy--Schwarz (since $g$ is $1$-bounded) we have $\mb{E}^{\log}_{p \in \mc{P}} |g(p)|^2 \ge \delta^2$. Therefore $\mb{E}^{\log}_{p \in \mc{P}} |\tilde g(p)|^2 \ge \delta^2 - 2\mb{E}^{\log}_{p \in \mc{P}} |\eps(p)| |g(p)| -  \mb{E}^{\log}_{p \in \mc{P}} |\eps(p)|^2\ge \delta^2/2$, using the $1$-boundedness of $g$ to estimate the second term. 
That is, 
\[  \mb{E}^{\log}_{p \in \mc{P}} \Big| \mb{E}_{n \in [N],p' \in \mc{P}'}^{\log}  p' \mbf{1}_{p' \mid n} f_1\big( \frac{n}{p'} + \lambda p p'\big) f_2(\lambda n p) \Big|^2 \ge \delta^2/2.\] Using Cauchy--Schwarz on the inner sum (and the $1$-boundedness of $f_2$) gives
\[  \mb{E}_{p \in \mc{P}}^{\log}  \mb{E}_{n \in [N]}^{\log} \Big|\mb{E}_{p' \in \mc{P}'}^{\log}  p' \mbf{1}_{p' \mid n} f_1\big( \frac{n}{p'} + \lambda p p'\big)  \Big|^2 \ge \delta^2/2.\] 
We now pass to a non-logarithmic average in the $p$ variable, on a suitable dyadic interval. To do this, first partition $[P_1, P_2]$ into intervals $I$ with $\frac{3}{2} \le \max(I)/\min(I) \le 2$. By averaging, there is some such $I$ for which 
\[  \mb{E}_{p \in \mc{P} \cap I}^{\log}  \mb{E}_{n \in [N]}^{\log} \Big|\mb{E}_{p' \in \mc{P}'}^{\log}  p' \mbf{1}_{p' \mid n} f_1\big( \frac{n}{p'} + \lambda p p'\big)  \Big|^2 \ge \delta^2/2.\] Let $X$ be such that $I \subset [X, 2X)$. We introduce the majorant $\tilde \Lambda$ from \cref{lem4point3}. Since the logarithmic weight $\frac{1}{p}$ varies by a factor at most $2$ on $I$, it follows that 
\[  \mb{E}_{x \in [X, 2X]}\tilde\Lambda(x) \mb{E}_{n \in [N]}^{\log} \Big|\mb{E}_{p' \in \mc{P}'}^{\log}  p' \mbf{1}_{p' \mid n} f_1\big( \frac{n}{p'} + \lambda x p'\big)  \Big|^2 \gg \delta^2.\]   Expanding out the square gives
\begin{equation}\label{eq301}  \mb{E}_{x \in [X, 2X)} \tilde\Lambda(x) \mb{E}_{n \in [N],p_1' \in \mc{P}',p_2' \in \mc{P}'}^{\log} p_1'p'_2 \mbf{1}_{[p'_1, p'_2] \mid n} f_1\big( \frac{n}{p'_1} + \lambda x p'_1\big) \overline{f_1\big( \frac{n}{p'_2} + \lambda x p'_2\big)} \gg \delta^2.\end{equation}  The next technical reduction is to replace the cutoff $\mbf{1}_{[p'_1, p'_2] \mid n}$ with $\mbf{1}_{p'_1 p'_2 \mid n}$, which we do using the fact that the elements of $\mc{P}'$ are mostly coprime due to \cref{coprimality-lem}. Let us justify this carefully. Since $f_1, f_2$ are $1$-bounded and $\mb{E}_{x\in [X,2X]}\wt{\Lambda}(x)\ll 1$, the error in making this switch in the LHS of \cref{eq301} is bounded up to a constant factor by 
\begin{equation}\label{eq302} \mb{E}_{n\in [N],p_1',p'_2\in \mc{P}'}^{\log}p_1'p_2' \big|\mbf{1}_{[p_1',p_2']\mid n} - \mbf{1}_{p_1'p_2'\mid n}\big|.\end{equation}
We have the pointwise bound $\big|\mbf{1}_{[p_1',p_2']\mid n} - \mbf{1}_{p_1'p_2' \mid n}\big| \le 2 \mbf{1}_{(p_1',p_2')\neq 1} \mbf{1}_{[p'_1, p'_2] \mid n}$ and therefore 
\[ \mb{E}^{\log}_{n \in [N]} \big|\mbf{1}_{[p_1',p_2']\mid n} - \mbf{1}_{p_1'p_2' \mid n}\big| \le 2 \mbf{1}_{(p'_1, p'_2) \ne 1} \mb{E}_{n \in [N]}^{\log} \mbf{1}_{[p'_1, p'_2] \mid n} \le \frac{4 \mbf{1}_{(p'_1, p'_2) \ne 1}}{[p'_1,p'_2]},\] using in the last step that $p'_1, p'_2$ are much smaller than $N$. It follows that \cref{eq302} is bounded above by $4\mb{E}_{p_1',p_2'\in \mc{P}'}^{\log}(p'_1,p'_2) \mbf{1}_{(p'_1, p'_2) \ne 1}$. Using the pointwise bound $(p'_1,p'_2) \mbf{1}_{(p'_1, p'_2) \ne 1} \le 2 ((p'_1, p'_2) - 1)$, this in turn is bounded by $8\mb{E}_{p_1',p'_2\in \mc{P}'}^{\log} ((p'_1,p'_2) - 1) = 8 \gamma(\mc{P}')$. By \cref{coprimality-lem}, we see that \cref{eq302} is bounded by $O(\delta^4)$. Therefore, as claimed, we may replace \cref{eq301} by
\begin{equation}\label{eq303} \Big| \mb{E}_{x \in [X, 2X)} \tilde\Lambda(x) \mb{E}_{n \in [N],p_1',p'_2 \in \mc{P}'}^{\log} p_1'p'_2 \mbf{1}_{p'_1 p'_2 \mid n} f_1\big( \frac{n}{p'_1} + \lambda x p'_1\big) \overline{f_1\big( \frac{n}{p'_2} + \lambda x p'_2\big)} \Big| \gg \delta^2.\end{equation} 
The reason for having replaced \cref{eq301} with \cref{eq303} is that we may now invoke \cref{lem:dilate} (with $q = p'_1 p'_2$) to conclude that 
\begin{equation}\label{eq304}  \Big|\mb{E}_{x \in [X, 2X)} \tilde\Lambda(x) \mb{E}_{n \in [N], p_1',p'_2 \in \mc{P}'}^{\log} f_1\big( n p'_2 + \lambda x p'_1\big) \overline{f_1\big( n p'_1 + \lambda x p'_2\big)}\Big| \gg \delta^2.\end{equation} 
We now apply \cref{lem4point3} with parameter $Q = \lfloor\delta^{-C_2}\rfloor$ and constant $c := 1/4C_1$. Observe that the required inequality
\begin{equation}\label{q-logx} Q \le \log X\end{equation} is true and follows from the choice of parameters, using here that $X \ge P_1$.

Let $\Lambda_{\per}$ be the $Q!$-periodic function as in that lemma. Our aim is to replace $\tilde\Lambda$ in \cref{eq304} by $\Lambda_{\per}$.

From \cref{eq304} and the triangle inequality, one of the following two statements holds:
\begin{equation}\label{eq304a}  \Big|\mb{E}_{x \in [X, 2X)} (\tilde\Lambda - \Lambda_{\per})(x) \mb{E}_{n \in [N], p_1',p'_2 \in \mc{P}'}^{\log} f_1\big( n p'_2 + \lambda x p'_1\big) \overline{f_1\big( n p'_1 + \lambda x p'_2\big)}\Big| \gg \delta^2,\end{equation} or
\begin{equation}\label{eq304b}  \Big|\mb{E}_{x \in [X, 2X)} \Lambda_{\per}(x) \mb{E}_{n \in [N], p_1',p'_2 \in \mc{P}'}^{\log} f_1\big( n p'_2 + \lambda x p'_1\big) \overline{f_1\big( n p'_1 + \lambda x p'_2\big)}\Big| \gg \delta^2.\end{equation} We analyse these two possibilities in turn. In the analysis we will use several times that 
\begin{equation}\label{crude-lam-lamper} \mb{E}_{x \in [X,2X)} |(\tilde\Lambda - \Lambda_{\per})(x)| \ll (\log Q)^{O(1)},\end{equation} which follows from \cref{lam-per-props} and the triangle inequality (since $\tilde\Lambda$ is non-negative).\vspace*{10pt}

\emph{Analysis of \cref{eq304a}.} We begin by dyadically localising the (two copies of) the set $\mc{P}'$. Recall that $\mc{P}' = \{p_1 \cdots p_k : p_i \in I_i\}$. Since $\max(I_i)/\min(I_i) \ge 10$, we can decompose each $I_i$ as a disjoint union of intervals $I_{i,j}$, each of the form $[Y, (1 + \eta_{i,j}) Y]$ for some $\eta_{i,j}$ satisfying $\frac{1}{8k} \le \eta_{i,j} \le \frac{1}{4k}$. We then have a corresponding decomposition $\mc{P}'= \bigcup_{j_1,\dots, j_k} \mc{P}'_{j_1,\dots j_k}$, where $\mc{P}'_{j_1,\dots, j_k} := \{ p_1 \cdots p_k : p_i \in I_{i, j_i}\}$. Note that, since $(1 + \frac{1}{4k})^k < 2$, each $\mc{P}'_{j_1,\dots, j_k}$ is contained in a dyadic interval. By averaging, there are $\vec{j} = (j_1,\dots, j_k)$ and $\vec{j}' = (j'_1,\dots, j'_k)$ such that 
\begin{equation}\label{eq306a}  \mb{E}_{p_1' \in \mc{P}'_{\vec{j}},p'_2\in \mc{P}_{\vec{j}'}'}^{\log}\Big|\mb{E}_{x \in [X, 2X)} (\tilde{\Lambda} - \Lambda_{\per})(x) \mb{E}_{n \in [N]}^{\log} f_1\big( n p'_2 + \lambda x p'_1\big) \overline{f_1\big( n p'_1 + \lambda x p'_2\big)} \Big|\gg \delta^2.\end{equation} For notational brevity, write $\mc{P}'_1 := \mc{P}'_{\vec{j}}$ and $\mc{P}'_2 := \mc{P}'_{\vec{j}'}$. 
As $\mc{P}_1',\mc{P}'_2$ are each contained in dyadic intervals, we can remove the logarithmic averaging to obtain
\begin{equation}\label{eq306b}  \mb{E}_{p_1'\in \mc{P}_1',p'_2 \in \mc{P}'_2}\Big|\mb{E}_{x \in [X, 2X)}(\tilde{\Lambda} - \Lambda_{\per})(x)  \mb{E}_{n \in [N]}^{\log} f_1\big( n p'_2 + \lambda x p'_1\big) \overline{f_1\big( n p'_1 + \lambda x p'_2\big)} \Big|\gg \delta^{2}.\end{equation}
The next several manipulations leading to \cref{eq306e} are straightforward and are aimed to replacing the logarithmic average over $n$ by an ordinary average on an appropriate subinterval. We first discard the contribution from small values of $N$. Set $N' := e^{(\log N)^{3/4}}$ (say). Writing \cref{eq306b} as
\[ \mb{E}_{p_1'\in \mc{P}_1',p'_2 \in \mc{P}'_2}\Big|\mb{E}_{x \in [X, 2X)} (\tilde{\Lambda} - \Lambda_{\per})(x) \sum_{n \in [N]} \frac{1}{n} f_1\big( n p'_2 + \lambda x p'_1\big) \overline{f_1\big( n p'_1 + \lambda x p'_2\big)} \Big|\gg \delta^{2} H_N,\](where $H_N$ is the harmonic sum), using \cref{crude-lam-lamper} we see that the contribution to the LHS from $n \le N'$ is bounded by $H_{N'} (\log Q)^{O(1)} < \delta^{10} H_N$, using here that $1/\delta \le \log \log N$.

Since $\frac{H_N}{H_N - H_{N'}} \approx 1$, it follows that we may replace \cref{eq306b} by
\[ \mb{E}_{p_1'\in \mc{P}_1',p'_2 \in \mc{P}'_2}\Big|\mb{E}_{x \in [X, 2X)}  (\tilde{\Lambda} - \Lambda_{\per})(x)\mb{E}_{n \in [N',N]}^{\log} f_1\big( n p'_2 + \lambda x p'_1\big) \overline{f_1\big( n p'_1 + \lambda x p'_2\big)} \Big|\gg \delta^{2}.\]
We now break $[N',N]$ into intervals $I$ whose lengths satisfy $e^{(\log N)^{1/2}} \le |I| \le 2e^{(\log N)^{1/2}}$. By pigeonhole there exists such an interval for which 
\[  \mb{E}_{p_1'\in \mc{P}_1',p_2' \in \mc{P}_2'}\Big|\mb{E}_{x \in [X, 2X)} (\tilde{\Lambda} - \Lambda_{\per})(x) \mb{E}_{n \in I}^{\log} f_1\big( n p'_2 + \lambda x p'_1\big) \overline{f_1\big( n p'_1 + \lambda x p'_2\big)} \Big|\gg \delta^{2}.\] The weight $1/n$ varies by at most $1 + O(|I| \cdot N'^{-1})$ on $I$ and so, using \cref{crude-lam-lamper}, we can justify replacing $\mb{E}^{\log}_{n \in I}$ with a uniform average $\mb{E}_{n \in I}$, thus obtaining
\begin{equation}\label{eq306e}  \mb{E}_{p_1'\in \mc{P}_1',p_2' \in \mc{P}'_2}\Big|\mb{E}_{x \in [X, 2X)} (\tilde{\Lambda} - \Lambda_{\per})(x) \mb{E}_{n \in I}f_1\big( n p'_2 + \lambda x p'_1\big) \overline{f_1\big( n p'_1 + \lambda x p'_2\big)} \Big|\gg \delta^{2}.\end{equation}
Our plan now is to use the decomposition $\tilde\Lambda - \Lambda_{\per} = \sum_i g_i + h$ from \cref{lem4point3} in order to obtain a contradiction from  \cref{eq306e}. To do this, we claim that for a general function $\psi : [X, 2X) \rightarrow \C$ we have 
\begin{align}\nonumber
\mb{E}_{p_1'\in \mc{P}_1',p'_2 \in \mc{P}'_2} \Big| \mb{E}_{x \in [X, 2X)} & \psi(x)  \mb{E}_{n \in I} f_1\big( n p'_2 + \lambda x p'_1\big) \overline{f_1\big( n p'_1 + \lambda x p'_2\big)}\Big| \\ & \ll \min \Big( \mb{E}_{x \in [X, 2X)} |\psi(x)|  ,  \frac{k}{X^c} \Vert \wh{\psi} \Vert_{\infty}^c \Vert \psi \Vert_{\infty}^{1 - c}  + (\log X)^{-C_2}\Vert \psi\Vert_{\infty}  \Big).\label{gen-avg-tau}
\end{align}
Here, $\wh{\psi}(\theta)  =  \sum_{x \in [X, 2X)} \psi(x) e(-\theta x)$. Assuming the claim for now, we see that the LHS of \cref{eq306e} is bounded above by
\[ \ll \frac{k}{X^c} \sum_i \Vert \wh{g_i} \Vert_{\infty}^c \Vert g_i \Vert_{\infty}^{1 - c} + (\log X)^{-C_2} \sum_i \Vert g_i \Vert_{\infty} + \mb{E}_{x \in [X, 2X)} |h(x)| \ll k Q^{-c/4}\] by \cref{small-h-ell1,sum-gi-gi-hat,gi-ell-infty}, assuming here that $C_2$ is sufficiently large and noting \cref{q-logx}. This contradicts \cref{eq306e}, recalling here that $Q = \lfloor\delta^{-C_2}\rfloor$, that $C_2$ is sufficiently large in terms of $C_1$, and additionally recalling here our assumption (in \cref{main-sec3}) that $k \le \delta^{-10}$. That is (assuming the claim \cref{gen-avg-tau}) we cannot have \cref{eq304a}, and therefore \cref{eq304b} holds.\vspace*{10pt}

\emph{Proof of claim \cref{gen-avg-tau}.} The first bound is trivial, but the second is a somewhat involved task. By homogeneity, we may assume that $\Vert \psi \Vert_{\infty} = 1$. Thus if the second bound in \cref{gen-avg-tau} does not hold, we have
\begin{equation}\label{gen-avg-tau-2}
\mb{E}_{p_1'\in \mc{P}_1',p'_2 \in \mc{P}'_2}\Big|\mb{E}_{x \in [X, 2X)}  \psi(x) \mb{E}_{n \in I} f_1\big( n p'_2 + \lambda x p'_1\big) \overline{f_1\big( n p'_1 + \lambda x p'_2\big)}\Big|\ge \tau/\tau_0,
\end{equation} where we are free to choose an absolute $\tau_0$ and $\tau := \frac{k}{X^c}\Vert \wh{\psi} \Vert_{\infty}^c  + (\log X)^{-C_2}$, thus in particular \begin{equation}\label{tau-interval} \tau \in [(\log X)^{-C_2},\tau_0].\end{equation} It therefore suffices to show that the assumption \cref{gen-avg-tau-2} and the inclusion \cref{tau-interval} imply that
\begin{equation}\label{large-psi-fourier}
\Vert \wh{\psi} \Vert_{\infty} = \sup_{\theta \in \R/\Z} |\wh{\psi}(\theta)| \ge (\tau/k)^{1/c}X ,
\end{equation}
since this immediately contradicts the definition of $\tau$. The remainder of the proof of claim \cref{gen-avg-tau} is devoted to this task.

Suppose that $\mc{P}'_{1} \subset [Y_1, 2Y_1]$ and $\mc{P}'_2 \subset [Y_2, 2Y_2]$, where $P'_1 \le Y_1, Y_2 \le P'_2$. Set $T_1 := \lfloor \tau^{C_1} X/Y_1 \rfloor$ and $T_2 := \lfloor \tau^{C_1} X/Y_2\rfloor$. Since $X \ge P_1 > (P'_2)^{10} \ge Y_i^{10}$ (by one of the assumptions of \cref{main-sec3}) and $\tau^{C_1} \ge (\log X)^{-C_1C_2}$ we have $T_1, T_2 \ge X^{1/2} \ge 1$. Let $t_1, t_2$ be integers with $|t_i| \le T_i$, and substitute $n := n' - \lambda p'_1 t_2 - \lambda p'_2 t_1$, $x := x' + p'_1 t_1 + p'_2 t_2$ in \cref{gen-avg-tau-2}. This gives
\begin{align}\nonumber
\mb{E}_{p_1'\in \mc{P}_1',p_2' \in \mc{P}_2'} &\Big|\mb{E}_{x' \in [X, 2X) -  p'_1 t_1 - p'_2 t_2}   \mb{E}_{n' \in I + \lambda (p'_1 t_2 +  p'_2 t_1)} f_1\big( n' p'_2 + \lambda x' p'_1 + \lambda(p_1'^2 - p_2'^2)t_1\big)  \\ & \times \overline{f_1\big( n' p'_1 + \lambda x' p'_2 + \lambda(p_2'^2 - p_1'^2)t_2\big)} \psi(x' +  p'_1 t_1 + p'_2 t_2)\Big|\ge \tau .\label{eq317}
\end{align}

Now observe that $|p'_1 t_1 + p'_2 t_2| \ll \tau^{C_1} X$, and also we have the crude bound 
\[ |\lambda (p'_1 t_2 +  p'_2 t_1)| \le |\lambda| P'_2 X \le e^{3 (\log N)^{1/4}} \ll \tau^{10} |I|,\] using here that all of $|\lambda|, P'_2, X$ are $\le e^{(\log N)^{1/4}}$, that $|I| \ge e^{(\log N)^{1/2}}$ and that $\tau \ge (\log X)^{-C_2} > (\log N)^{-C_2}$. It follows using \cref{avg-shifts} that for each fixed $p'_1, p'_2$ we may replace the $x'$-average in \cref{eq317} by $\mb{E}_{x \in [X, 2X)}$, and the $n'$-average by $\mb{E}_{n' \in I}$, at the cost of changing the inner sum in \cref{eq317} by $O(\tau^2)$. Doing this, averaging over $t_1, t_2$ and dropping the dashes on $x', n'$ for clarity, we obtain 
\begin{align*}
\mb{E}_{p_1'\in \mc{P}_1', p_2'\in \mc{P}_2'} \Big|&\mb{E}_{t_1\in [T_1], t_2 \in [T_2],x \in [X, 2X),n \in I}  f_1\big( n p'_2 + \lambda x p'_1 + \lambda(p_1'^2 - p_2'^2)t_1\big) \\ &\qquad\qquad \times \overline{f_1\big( n p'_1 + \lambda x p'_2 + \lambda(p_2'^2 - p_1'^2)t_2\big)}  \psi(x +  p'_1 t_1+ p'_2 t_2 )\Big|\ge \tau/2.
\end{align*}
Therefore there exists $n$ such that 
\begin{align*}
\mb{E}_{p_1'\in \mc{P}_1', p_2'\in \mc{P}_2'} \Big|&\mb{E}_{t_1\in [T_1], t_2 \in [T_2],x \in [X, 2X)}  f_1\big( n p'_2 + \lambda x p'_1 + \lambda(p_1'^2 - p_2'^2)t_1 \big) \\ &\qquad\qquad\qquad\times \overline{f_1\big( n p'_1 + \lambda x p'_2 + \lambda(p_2'^2 - p_1'^2)t_2\big)}  \psi(x +  p'_1 t_1 + p'_2 t_2 )\Big|\ge \tau/2.
\end{align*}
This implies that
\begin{align*}
\mb{E}_{p_1'\in \mc{P}_1', p_2'\in \mc{P}_2', t_1\in [T_1], t_2 \in [T_2], x \in [X, 2X)} F_1(p_1',p_2',x,t_1)F_2(p_1',p_2',x,t_2)\psi(x + p'_1 t_1 + p'_2 t_2 )\ge \tau/2 
\end{align*}
with $F_i$ being $1$-bounded functions; here we have absorbed the absolute value as a unit complex number into $F_1(p_1',p_2',x,t_1)$. We now apply Cauchy--Schwarz twice, and replace the dummy variable $x$ by $n$, to obtain that 
\[
\mb{E}_{p_1'\in \mc{P}_1',p_2'\in \mc{P}_2', t_1,t_1'\in [T_1], t_2, t'_2 \in [T_2], n \in [X, 2X)} \Delta_{p'_1 (t_1,t_1')}\Delta_{p_2'(t_2,t_2')}\psi(n)\ge (\tau/2)^4.
\]
(The notation used here is described in \cref{notation-sec}.) 

The expression on the LHS is the same as the one in \cref{lem26-assump-2}, with $S_i = \mc{P}'_i$. In order to apply \cref{lem:input-concat-2-iter}, we need the sets $\mc{P}'_1,\mc{P}'_2$ to have suitable diophantine properties. Such a statement is precisely the content of \cref{lem:crit-estimate}. To see this, recall that by definition we have $\mc{P}'_1 = \{ p_1 \cdots p_k : p_i \in I'_i\}$, where each interval $I'_i$ has the form $[M_i, (1 + \eta_i) M_i)$ for some $\eta_i \in (1/8k,1/4k)$ and some $M_i$, which we may assume to be the smallest prime $p_i$ in $I'_i$. Note that we always have $P'_1 \le M_i \le P'_2$, and $Y_1 \le M_1 \cdots M_k$ since $M_1 \cdots M_k \in \mc{P}'_1$ and $\mc{P}'_1 \subset [Y_1, 2Y_1]$. We now apply \cref{lem:crit-estimate} with $j = 1$. The required condition $\min_i M_i > Q^{L_1}$ in that lemma follows using \cref{p1p2k}. Thus \cref{lem:crit-estimate} gives that $\mc{P}'_1$ is $(L_1, k, Y_1)$-diophantine, for some absolute constant $L_1$. Similarly, $\mc{P}'_2$ is $(L_1, k, Y_2)$-diophantine.

We may now apply \cref{lem:input-concat-2-iter} with $S_i = \mc{P}'_i$ for $i = 1,2$, $\delta = (\tau/2)^4$, $L = L_1$, $L' = k$ and $D_i = Y_i$ for $i = 1,2$. To apply that lemma we need to verify, for $i = 1,2$, the three conditions $D_i, T_i \ge (L'/\delta)^{C_{\operatorname{\ref{lem:input-concat-2-iter}}}L^2}$ and $T_i D_i \le (L'/\delta)^{C_{\operatorname{\ref{lem:input-concat-2-iter}}}L^2}X$. The first condition holds comfortably using $Y_i \ge P'_1$ and the choice of parameters. The second condition holds even more comfortably using $T_i \ge X^{1/2} \ge P_1^{1/2}$ and the choice of parameters. Finally, the third condition holds using that $T_i D_i \asymp \tau^{C_1} X$, provided $C_1$ is large enough; larger than $4L_1C_{\operatorname{\ref{lem:input-concat-2-iter}}}$ is sufficient.

The conclusion of \cref{lem:input-concat-2-iter} gives that for any $H_1, H_2 \in \N$ with $H_i \le(\tau/k)^{2C_1} X$, there are $q_1,q_2 \in \N$, $q_i\le (k/\tau)^{O(1)}$ such that 
\begin{equation}\label{gp-prelim} \mb{E}_{n \in [2X], h_1, h'_1 \in [H_1], h_2, h'_2 \in [H_2]} \Delta_{q_1(h_1, h'_1)} \Delta_{q_2(h_2, h'_2)} \psi(n) \gg (\tau/k)^{O(1)}.\end{equation}
Set $H_1 = H_2 = H := \lfloor (\tau/k)^{2C_1} X\rfloor$.
The expression on the left in \cref{gp-prelim} is closely related to the Gowers $U^2$-norm of $\psi$ (or more accurately a Gowers--Peluse norm; see \cite{Pel20} where they are called ``Gowers box norms''). Rather than appeal to any general theory of such norms, we proceed with a direct analysis using the Fourier transform. By the Fourier expansion $\psi(n) = \int_{\R/\Z} \wh{\psi}(\theta) e(n\theta) d\theta$, \cref{gp-prelim} is 
\begin{align*} \int \wh{\psi}(\theta_1)\overline{\wh{\psi}(\theta_2)} & \overline{ \wh{\psi}(\theta_3)}\wh{\psi}(\theta_4) \wh{\mu_{[2X]}}(-\theta_1 + \theta_2 + \theta_3 - \theta_4) \wh{\mu_{[H]}}(q_1 (-\theta_1 + \theta_3))\wh{\mu_{[H]}}(q_1 (-\theta_2 + \theta_4)))\\ & \qquad\qquad \times \wh{\mu_{[H]}}(q_2(-\theta_1 + \theta_2)) \wh{\mu_{[H]}}(q_2(-\theta_3 + \theta_4))d\theta_1 d\theta_2 d\theta_3 d\theta_4 \gg (\tau/k)^{O(1)}.\end{align*} Here, $\mu_{[M]}$ denotes the normalised probability measure on $[M]$. By AM-GM and the pointwise bound $|\wh{\mu_{[2X]}}|\le 1$ we have that 
\begin{align*} \int \sum_{j=1}^4|\wh{\psi}(\theta_j)|^4& \Big|\wh{\mu_{[H]}}(q_1 (-\theta_1 + \theta_3))\wh{\mu_{[H]}}(q_1 (-\theta_2 + \theta_4)))\\ & \qquad\qquad \times \wh{\mu_{[H]}}(q_2(-\theta_1 + \theta_2))\wh{\mu_{[H]}}(q_2(-\theta_3 + \theta_4))\Big|d\theta_1 d\theta_2 d\theta_3 d\theta_4 \gg (\tau/k)^{O(1)}.\end{align*}
Substitute $\theta'_i = -\theta_i + t$, for $t \in \R/\Z$, and integrate over $t$. This gives (dropping the dashes) 
\begin{align}\nonumber \Big(4\int |\wh{\psi}(t)|^4~dt\Big)& \int  \Big|\wh{\mu_{[H]}}(q_1 (\theta_1 - \theta_3))\wh{\mu_{[H]}}(q_1 (\theta_2 - \theta_4)))\\ & \qquad\qquad \times \wh{\mu_{[H]}}(q_2(\theta_1 - \theta_2))\wh{\mu_{[H]}}(q_2(\theta_3 - \theta_4))\Big|d\theta_1 d\theta_2 d\theta_3 d\theta_4\gg (\tau/k)^{O(1)}.\label{fourier424-c}\end{align}
We claim that 
\begin{equation}\label{kernel-bd-claim} \int \Big|\wh{\mu_{[H]}}(q_1 (\theta_1 - \theta_3))\wh{\mu_{[H]}}(q_1 (\theta_2 - \theta_4)))\wh{\mu_{[H]}}(q_2(\theta_1 - \theta_2))\wh{\mu_{[H]}}(q_2(\theta_3 - \theta_4))\Big|d\theta_1 d\theta_2 d\theta_3 d\theta_4\ll  H^{-3}.
\end{equation}
By AM-GM and symmetry it suffices to prove that
\[ \int \Big|\wh{\mu_{[H]}}(q_1 (\theta_1 - \theta_3)) \wh{\mu_{[H]}}(q_1 (\theta_2 - \theta_4))) \wh{\mu_{[H]}}(q_2(\theta_1 - \theta_2))\Big|^{4/3} d\theta_1 d\theta_2 d\theta_3 d\theta_4\ll  H^{-3}.
\] The triple $(q_1(\theta_1 - \theta_3),  q_1(\theta_2 - \theta_4), q_2(\theta_1 - \theta_2))$ ranges uniformly over $(\R/\Z)^3$ as $(\theta_1,\theta_2,\theta_3,\theta_4)$ ranges over $(\R/\Z)^4$ and so it is enough to show that $\int|\wh{\mu_{[H]}}(\theta)|^{4/3} d\theta\ll  H^{-1}$. This, however, follows immediately using the bound $|\wh{\mu_{[H]}}(\theta)| \ll \min \big(1, H^{-1} \Vert \theta\Vert_{\R/\Z}^{-1}\big)$. 
The claim \cref{kernel-bd-claim} is therefore proven. From this and \cref{fourier424-c} we immediately have $\int |\wh{\psi}(t)|^4 d\theta \ge (\tau/k)^{O(1)} H^3 \gg (\tau/k)^{7C_1} X^3$ (if $C_1$ is sufficiently large).
Since $\int |\wh{\psi}(t)|^2 dt\ll X$ by Parseval, it follows that $\Vert \wh{\psi}\Vert_{\infty} \gg (\tau/k)^{7C_1/2} X$ and so $\Vert \wh{\psi}\Vert_{\infty} \ge (\tau/k)^{4C_1} X = (\tau/k)^{1/c}X$. In this last step we used the fact \cref{tau-interval} that $\tau \le \tau_0$; what is written is then true if $\tau_0$ is chosen sufficiently small. This completes the proof that the claims \cref{gen-avg-tau-2,tau-interval} imply \cref{large-psi-fourier}, and hence finishes the proof of the claim \cref{gen-avg-tau}.\vspace*{10pt}

As explained just before the statement of claim \cref{gen-avg-tau}, it now follows that \cref{eq304b} holds. The remainder of the proof of \cref{main-sec3} consists of the analysis of this case.\vspace*{10pt}

\emph{Analysis of \cref{eq304b}.} We first recall the statement, which is (after a mild reordering of the averaging operators)
\begin{equation}\label{eq304b-rpt} \Big| \mb{E}^{\log}_{p_1',p_2' \in \mc{P}'} \mb{E}_{h \in[X, 2X)} \Lambda_{\per}(h) \mb{E}_{n \in [N]}^{\log} f_1\big( n p'_2 + \lambda h p'_1\big) \overline{f_1\big( n p'_1 + \lambda h p'_2\big)}\Big|  \gg \delta^2.\end{equation}
The advantage of having the function $\Lambda_{\per}$ in place of $\tilde{\Lambda}$ is that the former is invariant under shifts by $Q!$. This is by construction (\cref{lem4point3}); recall here that $Q = \lfloor \delta^{-C_2}\rfloor$. For fixed $p'_1, p'_2$, in the inner average over $h$ and $n$ in \cref{eq304b-rpt} we substitute $n := n' - Q!\lambda p'_2 t$ and $h := h' + Q!p'_1 t$ for some $t \in \Z$ and then average over all $t \in [P_1^{1/2}]$. By the periodicity of $\Lambda_{\per}$ we obtain
\begin{align}\nonumber \Big| \mb{E}_{t \in [P_1^{1/2}]} \mb{E}^{\log}_{p_1',p_2' \in \mc{P}'}\mb{E}_{h' \in[X, 2X) - Q!p'_1 t} \Lambda_{\per}(h') \mb{E}_{n' \in [N] + Q! \lambda p'_2 t}^{\log} & f_1\big( n' p'_2 + \lambda h' p'_1+ \lambda t Q!(p^{\prime 2}_1 - p^{\prime 2}_2)\big) \\ & \times \overline{f_1\big( n' p'_1 + \lambda h' p'_2\big)}\Big|  \gg \delta^2.\label{eq343}\end{align}
Fix $t, p'_1, p'_2$. By \cref{log-avg-shift}, crude bounds for the parameters, and \cref{lam-per-props}, the error in replacing the average over $n'$ by $\mb{E}^{\log}_{n' \in [N]}$ is 
\[ \ll \frac{\log (Q! |\lambda| P'_2 P_1^{1/2})}{\log N}\mb{E}_{h'} |\Lambda_{\per}(h')| \ll (\log N)^{-1/2}\mb{E}_{h'} |\Lambda_{\per}(h') |\ll (\log N)^{-1/2}\log(\frac{1}{\delta})^{O(1)} \lll \delta^{10}, \] so we may make this replacement without affecting \cref{eq343}. Moreover, by applying \cref{ord-avg-shift} and the bound $\Vert \Lambda_{\per} \Vert_{\infty} \le Q^2$, the error in then replacing the average over $h'$ by $\mb{E}_{h' \in [X, 2X)}$ is $\ll Q^2 \cdot \frac{Q! P'_1 P_1^{1/2}}{X}  \ll P_1^{-1/4} \ll \delta^{10}$, so we may again make the replacement without affecting \cref{eq343}. (In the chain of inequalities here we used that $P'_1 \le P'_2 \le P_1^{1/10}$, that $X \ge P_1$ and that $P_1$ is much larger than fixed powers of $Q!$ and $\delta^{-1}$, cf. remarks in \cref{subsec41}.) Having made these two replacements we drop the dashes on $n',h'$ for clarity, thereby arriving at
\[ \Big| \mb{E}_{t \in [P_1^{1/2}]} \mb{E}^{\log}_{p'_1,p_2' \in \mc{P}'}  \mb{E}_{h \in  [X, 2X)} \Lambda_{\per}(h) \mb{E}_{n \in [N]}^{\log} f_1\big( n p'_2 + \lambda h p'_1 + \lambda t Q! (p^{\prime 2}_1 - p^{\prime 2}_2)\big) \overline{f_1\big( n p'_1 + \lambda h p'_2\big)}\Big|  \gg \delta^2.\]
By the triangle inequality, we obtain
\[  \mb{E}_{h \in  [X, 2X)} |\Lambda_{\per}(h) |\mb{E}_{n \in [N],p_1' ,p_2' \in \mc{P}'}^{\log}\Big| \mb{E}_{t \in [P_1^{1/2}]} f_1\big( n p'_2 + \lambda h p'_1 + \lambda t Q! (p^{\prime 2}_1 - p^{\prime 2}_2)\big)\Big|  \gg \delta^2.\]
Applying Cauchy--Schwarz, we obtain
\begin{align*} \mb{E}_{h \in  [X, 2X)} |\Lambda_{\per}(h)| \mb{E}_{t,t' \in [P_1^{1/2}]}  \mb{E}_{n \in [N], p_1',p'_2 \in \mc{P}'}^{\log}  f_1\big(&  n p'_2 + \lambda h p'_1  + \lambda t Q! (p^{\prime 2}_1 - p^{\prime 2}_2)\big)\\ &\times \overline{f_1\big( n p'_2 + \lambda h p'_1 + \lambda t' Q! (p^{\prime 2}_1 - p^{\prime 2}_2)\big)}  \gg \delta^4.\end{align*}
Using the pointwise bound $\mbf{1}_{(p'_1, p'_2) \ne 1} \le (p'_1, p'_2) - 1$ and the fact (\cref{coprimality-lem}) that $\gamma(\mc{P}') \le 
\delta^{4 + \eps_0/2}$, as well as the bound $\mb{E}_{h \in [X,2X)} |\Lambda_{\per}(h)| \ll \log^{O(1)}(1/\delta)$ (see \cref{lam-per-props}), we see that the contribution from pairs with $(p'_1, p'_2) \ne 1$ can be ignored. Thus
\begin{align*}   \mb{E}_{h \in [X, 2X)} |\Lambda_{\per}(h)| \mb{E}_{t,t' \in [P_1^{1/2}]}  \mb{E}_{n \in [N], p_1',p'_2 \in \mc{P}'}^{\log} \mbf{1}_{(p'_1, p'_2) = 1} & f_1\big( n p'_2 + \lambda h p'_1 + \lambda t Q! (p^{\prime 2}_1 - p^{\prime 2}_2)\big) \\ & \times \overline{f_1\big( n p'_2 + \lambda h p'_1 + \lambda t' Q! (p^{\prime 2}_1 - p^{\prime 2}_2)\big)}  \gg \delta^4.\end{align*}

Since $\Lambda_{\per}$ is invariant under shifts by $Q!$, we may introduce an additional average obtaining
\begin{align*} & \mb{E}_{h \in [X, 2X),h' \in [X'], t,t' \in [P_1^{1/2}]} |\Lambda_{\per}(h)|\mb{E}_{n \in [N], p_1',p'_2 \in \mc{P}'}^{\log} \mbf{1}_{(p'_1, p'_2) = 1}  \\ & \times f_1\big( n p'_2 + \lambda (h + Q!h') p'_1 + \lambda t Q! (p^{\prime 2}_1 - p^{\prime 2}_2)\big)  \overline{f_1\big( n p'_2 + \lambda (h + Q!h') p'_1 + \lambda t' Q! (p^{\prime 2}_1 - p^{\prime 2}_2)\big)} \gg \delta^4,\end{align*} where here $X' := \lfloor \delta^{5} X/Q!\rfloor$; note that $X'$ is much larger than 1 by the choice of parameters. Apart from the invariance of $\Lambda_{\per}$ under translation by $Q!$, the key point here is that, for each fixed $h'$, the shifted average differs from the original one by at most $X^{-1} \sum_{h \in [X, X + X'Q!]}|\Lambda_{\per}(h)| \ll \frac{X'Q!}{X} \log(1/\delta)^{O(1)}$ by \cref{lam-short-interval} (and a similar term corresponding to the edge effects near $2X$).

In the display above, consider the average over $n, h'$ (for fixed $h, t, t', p'_1, p'_2$). The point now is that, from the point of view of logarithmic averages, $n p'_2 + \lambda Q! h' p'_1$ may be regarded as essentially just varying over $[N]$. More precisely, applying \cref{frobenius-coin} with $q = p'_2$, $b = \lambda Q! p'_1$, $H := X' = \lfloor \delta^5 X/Q!\rfloor$ and $f(x) := f_1(x + \lambda h p'_1 + \lambda t Q!(p^{\prime 2}_1 - p^{\prime 2}_2))\overline{f_1(x + \lambda h p'_1 + \lambda t' Q!(p^{\prime 2}_1 - p^{\prime 2}_2))}$, we may replace the above with
\begin{align}\nonumber  \mb{E}_{h \in [X, 2X)} |\Lambda_{\per}(h)| \mb{E}_{t,t' \in [P_1^{1/2}]}  \mb{E}_{n \in [N], p_1',p'_2 \in \mc{P}'}^{\log} & \mbf{1}_{(p'_1, p'_2) = 1}  f_1\big( n + \lambda h p'_1 + \lambda t Q! (p^{\prime 2}_1 - p^{\prime 2}_2)\big)  \\ & \times \overline{f_1\big( n  + \lambda h p'_1 + \lambda t' Q! (p^{\prime 2}_1 - p^{\prime 2}_2)\big)}  \gg \delta^4.\label{eq4811}\end{align}
Let us comment on the application of \cref{frobenius-coin}. First, we used that $q = p'_2$ and $b = \lambda Q! p'_1$ are coprime. That $(p'_2,\lambda) = 1$  follows from the assumption that all prime factors of $\lambda$ are less than $P'_1$, and that $(p'_2, Q!) = 1$ follows using that $P'_1$ is much larger than $\delta^{-C_2}$. The error terms $O\big(\frac{\log q + \log bh}{\log N}\big)$ and $O\big(\frac{q}{H}\big)$ resulting from the application of \cref{frobenius-coin} are all $\ll \delta^{10}$ by simple verifications using the choice of parameters, the key point being that $H > P_1^{1/2}$ is much larger than $q$, but $bH < P_2^2$ is much smaller than $N$.

Applying \cref{log-avg-shift}, we may remove the $\lambda h p'_1$ shifts in \cref{eq4811}, allowing us to decouple the average over $h$ and thus obtain via another application of \cref{lam-per-props} that
\begin{equation}\label{eq339}  \mb{E}_{t,t' \in [P_1^{1/2}]} \mb{E}_{n \in [N], p_1',p'_2 \in \mc{P}'}^{\log} \mbf{1}_{(p'_1, p'_2) = 1} f_1\big( n + \lambda t Q! (p^{\prime 2}_1 - p^{\prime 2}_2)\big)\overline{f_1\big( n  + \lambda t' Q! (p^{\prime 2}_1 - p^{\prime 2}_2)\big)}  \gg \delta^{4 + \eps_0/4}.\end{equation}

We may remove the condition $(p'_1, p'_2) = 1$ (losing a further factor of 2 in the implicit constant) exactly as before, obtaining
\[  \mb{E}_{t,t' \in [P_1^{1/2}]}  \mb{E}_{n \in [N], p_1',p'_2 \in \mc{P}'}^{\log} f_1\big( n + \lambda t Q! (p^{\prime 2}_1 - p^{\prime 2}_2)\big)\overline{f_1\big( n  + \lambda t' Q! (p^{\prime 2}_1 - p^{\prime 2}_2)\big)}  \gg \delta^{4 + \eps_0/4} > \delta^5.\]
To analyse this, we will eventually use the diophantine nature of suitable sets $\{ (p')^2 : p' \in \mc{P}'\}$, applying \cref{lem:crit-estimate} in the case $j = 2$. To prepare the ground, we must again foliate into appropriate `subdyadic products' as we did in the analysis of \cref{eq304a} leading to \cref{eq306a}. With notation exactly the same as in that analysis, we may locate $\mc{P}'_1 := \mc{P}'_{\vec{j}}$ and $\mc{P}'_2 := \mc{P}'_{\vec{j}'}$ such that 
\[  \mb{E}_{t,t' \in [P_1^{1/2}]}  \mb{E}^{\log}_{n \in [N]} \mb{E}_{p_1' \in \mc{P}'_1, p'_2 \in \mc{P}'_2} f_1\big( n + \lambda t Q! (p^{\prime 2}_1 - p^{\prime 2}_2)\big)\overline{f_1\big( n  + \lambda t' Q! (p^{\prime 2}_1 - p^{\prime 2}_2)\big)}  \ge \delta^5.\] Note here that we were able to replace the logarithmic average over the $p'_i$ variables by a uniform average since these are now dyadically localised, and each $t,t'$-average is nonnegative. Suppose that $\mc{P}'_{1} \subset [Y_1, 2Y_1]$ and $\mc{P}'_2 \subset [Y_2, 2Y_2]$, where $P'_1 \le Y_1, Y_2 \le P'_2$. Without loss of generality, $Y_1 \ge Y_2$. Pigeonholing in $p'_2$, we see that there is some $p'_2$ such that 
\[  \mb{E}_{t,t' \in [P_1^{1/2}]}  \mb{E}_{n \in [N]}^{\log} \mb{E}_{p_1' \in \mc{P}'_1} f_1\big( n + \lambda t Q! (p^{\prime 2}_1 - p^{\prime 2}_2)\big)\overline{f_1\big( n  + \lambda t' Q! (p^{\prime 2}_1 - p^{\prime 2}_2)\big)}  \ge \delta^5.\]

By \cref{residue-split} with modulus $q = \lambda Q!$, this gives
\[ \mb{E}_{a \in \{0,1,\dots, \lambda Q! - 1\}} \mb{E}^{\log}_{n \in [N]}\mb{E}_{p_1' \in \mc{P}'_1, t,t' \in [P_1^{1/2}]}f_{1,a}\big( n +  t (p^{\prime 2}_1 - p^{\prime 2}_2)\big)\overline{f_{1,a}\big( n  + t' (p^{\prime 2}_1 - p^{\prime 2}_2)\big)} \ge \delta^6 \]
 where $f_{1,a}(n) := f_1(\lambda Q! n + a)$.   
 For each fixed $a$, the inner average is of the form \cref{lem26-assump}, with $S := \{ p^{\prime 2}_1 - p^{\prime 2}_2 : p'_1 \in \mc{P}'_1\}$ and $\delta$ replaced by $\delta^6$. We showed in \cref{lem:crit-estimate} (with $j = 2$) that $S + p^{\prime 2}_2 = \{ p^{\prime 2}_1 : p'_1 \in \mc{P}'_1\}$ is $(L_2, k, Y_1^2)$-Diophantine (the condition $\min_i M_i > Q^{L_2}$ in that lemma follows using \cref{p1p2k}), and so by translation invariance of the notion of diophantine, the same is true of $S$. Observe that $S \subset [-4Y_1^2, 4Y_1^2]$. Thus we may aim to apply \cref{lem:input-concat} with $S = \{ p_1^{\prime 2} - p_2^{\prime 2} : p'_1 \in \mc{P}'_1\}$, $T := \lfloor P_1^{1/2}\rfloor$, $(L, L', D) = (L_2, k, Y_1^2)$, and $\delta$ replaced by $\delta^6$. There are three conditions to be checked, namely that $D, T \ge (L'/\delta^6)^{8L} = (k/\delta^6)^{8L_2}$, and that $\frac{\log TD}{\log N} \le (\delta^6/k)^{50L_2}$. 

 The first condition, involving $D = Y_1^2$, is immediate from $Y_1 \ge P'_1$ and the parameter hierarchy. The second condition, involving $T = \lfloor P_1^{1/2}\rfloor$, is also immediate. For the third condition note that $T Y_1^2 \ge T \ge P_1^{1/2}$ and $(\delta^6/k)^{50L_2}$ is much smaller than $P_1^{1/4}$.
 
 Thus the appeal to \cref{lem:input-concat} is indeed valid, and we are free to take any $H = P_1^{1/4}$ in this application.
Recalling that $k \le \delta^{-10}$, the conclusion of \cref{lem:input-concat} that for each $a$ there is $q_a \le (k/\delta)^{O(1)} \le \delta^{-O(1)}$ such that $\Vert f_{1,a} \Vert_{U^1_{\log}[N; q_a, P_1^{1/4}]} \gg \delta^{O(1)}$.
 By pigeonhole there is a set of $\ge \delta^{O(1)} \lambda Q!$ values of $a$ such that $q_a$ does not depend on $a$. Denote this common value by $q$ (which is of course not the same quantity as in the application of \cref{residue-split} above). It follows that 
 \[ \mb{E}_{a \in \{0,1,\dots, \lambda Q! - 1\}}\Vert f_{1,a} \Vert^2_{U^1_{\log}[N; q_a, H]} \gg \delta^{O(1)} ,\] that is to say
\[  \mb{E}_{a \in \{0,1,\dots, \lambda Q! - 1\}} \mb{E}_{n \in [N]}^{\log}\mb{E}_{h,h' \in [P_1^{1/4}]}f_{1,a}(n + qh)\overline{f_{1,a}(n + qh')} \ge \delta^{O(1)}.\]
A further application of \cref{residue-split} then yields 
\[ \mb{E}_{n \in [N]}^{\log}\mb{E}_{h,h' \in [P_1^{1/4}]} f_1(n + \lambda h qQ!) \overline{f_1(n + \lambda h' q Q!)} \gg \delta^{O(1)},\] which is the statement
\[ \Vert f_1 \Vert^2_{U^1_{\log}[N; \lambda q Q!, P_1^{1/4}]} \gg \delta^{O(1)}.\]
Finally, let $H \le P_1^{1/8}$ be as in the statement of \cref{main-sec3}. Set $C := 2C_2$ and $V := \lfloor \delta^{-C} \rfloor !$. Note that $q Q! \mid V$. Therefore by \cref{gp-compar} we have 
\[ \Vert f_1 \Vert_{U^1_{\log}[N; \lambda V; H]} \ge \Vert f_1 \Vert_{U^1_{\log}[N; \lambda q Q!, P_1^{1/4}]} - O\Big(\frac{\log |P_1^{1/4} \lambda q Q!|}{\log N}\Big) - O\Big(\frac{H V}{P_1^{1/4} q Q!}\Big) \gg \delta^{O(1)},\]
where the error terms can be estimated crudely bearing in mind the comments in \cref{subsec41} (essentially, $P_1$ is much smaller than $N$ but much larger than all other variables).
This concludes the proof of \cref{main-sec3}.
\end{proof}

\section{Averaging projections and orthogonality}

In the introduction we discussed certain `projection' operators $\Pi^{\sml},\Pi^{\lrg}$. In this section we introduce the general class of such operators and establish some of their basic properties.

\begin{definition} \label{proj-def} Let $f : \Z \rightarrow \C$ be a function. Suppose that $q, H \in \N$. Then we define 
\[\Pi_{q,H}f(n) := \mb{E}_{h,h'\in [H]}f(n + q(h - h')).\]
\end{definition}

Whilst we informally think of these maps as projections, this is not quite accurate as $\Pi_{q,H} \Pi_{q,H}f\neq \Pi_{q,H} f$. The first observation we require is that $\Pi_{q,H} f$ has an almost periodicity property.

\begin{lemma}\label{almost-per-proj}
Let $f : \Z \rightarrow \C$ be a 1-bounded function. Let $q, H \in \N$. Then, for any $h$ we have
\[ \Pi_{q, H} f(n + qh) =  \Pi_{q,H}f(n) + O(\frac{|h|}{H}).\]
\end{lemma}
\begin{proof}
 The LHS may be expanded as $\mb{E}_{h_1, h'_1 \in [H]} f ( n + q(h + h_1 - h'_1))$. The result then follows from \cref{ord-avg-shift}.
\end{proof}

A crucial feature of the maps $\Pi_{q,H}$ is that they essentially preserve the $U^1_{\log}$-norms (see \cref{gp-local-def}). Indeed we have the following lemma.
\begin{lemma}\label{lem:proj-check}
Let $q\in \N$ and $H'\le H$. Then for $f:\Z\to \C$ which is $1$-bounded, we have that $\Vert \Pi_{q,H'}f  - f \Vert_{U^1_{\log}[N; q, H]} \ll H'/H$.
\end{lemma}
\begin{proof}
First recall that by definition \cref{gp-local-def} we have
\begin{equation}\label{norm-u1-def} \Vert g \Vert_{U^1_{\log}[N;q, H]}^2 = \mb{E}_{n \in [N]}^{\log}\big|\mb{E}_{h \in [H]}g(n + hq)\big|^2.\end{equation}
Note that by \cref{ord-avg-shift} we have 
\[\mb{E}_{h \in [H]}\Pi_{q,H'}f(n + hq) =  \mb{E}_{h \in [H], h'_1, h'_2 \in [H'] }f(n + q(h+h'_1-h'_2)) = \mb{E}_{h\in [H]}f(n + hq) + O\Big(\frac{H'}{H}\Big).\]
The desired result follows immediately upon taking $g = f - \Pi_{q,H'} f$ in \cref{norm-u1-def}.
\end{proof}

We next require an approximate Pythagoras relation for projections $\Pi_{H,q}, \Pi_{H', q'}$.

\begin{lemma}\label{lem:tel-help} Let $q, q', H, H'$ be parameters with $q \mid q'$ and $H' \le H$. Let $f : \Z \rightarrow \C$ be a $1$-bounded function.
We have that 
\[\mb{E}_{n \in [N]}^{\log}\big|\Pi_{q',H'}f(n) - \Pi_{q,H}f(n)\big|^2 \le \mb{E}_{n \in [N]}^{\log}\big|\Pi_{q',H'}f(n)\big|^2 - \mb{E}_{n \in [N]}^{\log}\big|\Pi_{q,H}f(n)\big|^2 + O\Big(\frac{\log q'H}{\log N} + \frac{q'H'}{qH}\Big).\]
\end{lemma}
\begin{proof} For brevity we write $\langle g_1, g_2 \rangle := \mb{E}_{n \in [N]}^{\log}g_1(n) \overline{g_2(n)}$ and $\Vert g \Vert^2 := \langle g, g\rangle = \mb{E}_{n \in [N]}^{\log} |g(n)|^2$.

We first expand the LHS as
\begin{equation}\label{52-lhs-exp} \Vert \Pi_{q',H'}f\Vert^2 + \Vert\Pi_{q,H}f\Vert^2- \langle \Pi_{q',H'}f, \Pi_{q,H}f\rangle - \overline{\langle \Pi_{q',H'}f, \Pi_{q,H}f\rangle }  .
\end{equation}
Expanding the definitions, we have \[
\langle \Pi_{q',H'} f, \Pi_{q,H} f\rangle  = \mb{E}_{n \in [N]}^{\log} \mb{E}_{h_1, h_2 \in [H], h'_1, h'_2 \in [H'] } f(n + q'(h'_1 - h'_2)) \overline{f(n + q(h_1 - h_2))}.
\]
Substitute $n = n' + qh_2 - q'h'_1$; then, dropping the dash on $n'$, we see from \cref{log-avg-shift} that this is 
\[ \mb{E}_{n \in [N]}^{\log} \mb{E}_{h_1, h_2 \in [H], h'_1, h'_2 \in [H'] } f(n + qh_2 - q'h'_2) \overline{f(n + qh_1 - q'h'_1))} + O\Big(\frac{\log q' H}{\log N}\Big),
\] which equals
\[ \mb{E}_{n \in [N]}^{\log} \big| \mb{E}_{h \in [H], h'\in [H'] } f(n + qh-q'h')\big|^2  + O\Big(\frac{\log q' H}{\log N}\Big).
\] 
Now by \cref{ord-avg-shift} (using here that $q \mid q'$) we have
\[ \mb{E}_{h \in [H], h'\in [H'] } f(n + q h-q'h') = \mb{E}_{h \in [H]} f(n + qh) + O\Big(\frac{q'H'}{qH}\Big).\]
Therefore, putting these observations together we obtain
\[ \langle \Pi_{q',H'} f, \Pi_{q,H} f\rangle 
  =\mb{E}_{n \in [N]}^{\log} \big|\mb{E}_{h \in [H]} f(n + qh)\big|^2 + O\Big(\frac{q'H'}{qH} + \frac{\log q' H}{\log N}\Big). \]
Taking complex conjugates and adding, we obtain
\[ \langle \Pi_{q',H'} f, \Pi_{q,H} f\rangle + \overline{\langle \Pi_{q',H'} f, \Pi_{q,H} f\rangle}
=2\mb{E}_{n \in [N]}^{\log} \big|\mb{E}_{h \in [H]} f(n + qh)\big|^2 + O\Big(\frac{q'H'}{qH} + \frac{\log q' H}{\log N}\Big). \]
Now by a further application of \cref{log-avg-shift}, 
\[ \mb{E}_{n \in [N]}^{\log} \big|\mb{E}_{h \in [H]} f(n + qh)\big|^2  = \mb{E}_{h' \in [H]}\mb{E}_{n \in [N]}^{\log} \big|\mb{E}_{h \in [H]} f(n + q(h - h'))\big|^2 + O\Big(\frac{\log qH}{\log N}\Big), \] and by Cauchy--Schwarz this is at least
\[ \mb{E}_{n \in [N]}^{\log}\big| \mb{E}_{h, h' \in [H]} f(n + q(h - h'))\big|^2 + O(\frac{\log qH}{\log N}) = \Vert \Pi_{q,H} f\Vert^2 + O\Big(\frac{\log qH}{\log N}\Big). \]
It follows that
\[ \langle \Pi_{q',H'} f, \Pi_{q,H} f\rangle + \overline{\langle \Pi_{q',H'} f, \Pi_{q,H} f\rangle} \ge 2\Vert  \Pi_{q,H} f\Vert^2  + O\Big(\frac{q'H'}{qH} + \frac{\log q' H}{\log N}\Big). \] Substituting in to \cref{52-lhs-exp} gives the lemma.\end{proof}

We now give the `maximal function' argument which was hinted at in the introduction where we explained how to move from \cref{ramsey-outcome} to \cref{ramsey-outcome-2}.

\begin{lemma}\label{maximal}
Let $f ,g : \N \rightarrow \C$ be non-negative $1$-bounded functions. Let $\delta \in (0,\frac{1}{2})$ and let $H,q$ be positive integer parameters with $\frac{\log Hq}{\log N} < c\delta^2$.  Suppose that $\mb{E}_{n \in [N]}^{\log} f(n) g(n) \ge \delta$. Then $\mb{E}_{n \in [N]}^{\log} (\Pi_{q,H} f)(n) g(n) \ge \delta^2/8$.
\end{lemma}
\begin{proof}
Write $\Pi = \Pi_{q,H}$ for brevity. Set $\eps := \delta/4$ and denote $h(n) := 1_{\Pi f(n) > \eps}$. Then since $0 \le f h \le 1$ and $(\Pi f) h \ge \eps h$ pointwise we have
\[ \mb{E}_{n \in [N]}^{\log} (\Pi f)(n) g(n) \ge \mb{E}_{n \in [N]}^{\log} f(n) (\Pi f)(n) g(n) h(n) \ge \eps \mb{E}_{n \in [N]}^{\log} f(n) h(n) g(n) .\]
Therefore we are done if we can show that $\mb{E}_{n \in [N]}^{\log} f(n) (1 - h(n)) \le \delta/2$, that is to say
\begin{equation}\label{want-55} \mb{E}_{n \in [N]}^{\log} f(n) 1_{\Pi f(n) \le \eps} \le \delta/2.\end{equation}
Write $F(n) := f(n) 1_{\Pi f(n) \le \eps}$. Since $F \le f$ pointwise, we have $\Pi F \le \Pi f$ pointwise, and so if $F(n) \ne 0$ then we have $\Pi F(n) \le \Pi f(n) \le \eps$. It follows that using \cref{log-avg-shift} and Cauchy--Schwarz that 
\begin{align*}
\big| \mb{E}_{n \in [N]}^{\log} F(n) \big|^2 & =  \big| \mb{E}_{n \in [N]}^{\log} \mb{E}_{h \in [H]} F(n + hq) \big|^2 + O\Big(\frac{\log Hq}{\log N}\Big) \\ & \le \mb{E}_{n \in [N]}^{\log} \big|\mb{E}_{h \in [H]} F(n + hq)\big|^2 +  O\Big(\frac{\log Hq}{\log N}\Big)\\ & = \mb{E}_{n \in [N]}^{\log} \mb{E}_{h, h' \in [H]} F(n + hq) F(n + h'q)  + O\Big(\frac{\log Hq}{\log N}\Big)\\ & \le \mb{E}_{n \in [N]}^{\log} \mb{E}_{h, h' \in [H]} F(n) F(n + (h - h')q) + O\Big(\frac{\log Hq}{\log N}\Big) \\ & = \mb{E}^{\log}_{n \in [N]} F(n) (\Pi F)(n) + O\Big(\frac{\log Hq}{\log N}\Big) \\ & \le \eps \mb{E}^{\log}_{n \in [N]} F(n) + \eps^2.
\end{align*}
It follows that $\mb{E}_{n \in [N]}^{\log} F(n) \le 2\eps$, so the claim \cref{want-55} follows due to the choice of $\eps$.
\end{proof}

We note a corollary under the same conditions which is good for taking averages, namely that for any $\eta$
\begin{equation}\label{avg-max} \mb{E}_{n \in [N]}^{\log} \Pi f(n) g(n) \ge \frac{\eta}{8} \mb{E}^{\log}_{n \in [N]} f(n) g(n) - \frac{\eta^2}{8}.\end{equation} Indeed, if we write $\delta := \mb{E}_{n \in [N]}^{\log} f(n) g(n)$ then \cref{avg-max} is trivial for $\delta \le \eta$, while for $\delta \ge \eta$ it follows from \cref{maximal}.

\section{Proof of the main theorem}\label{sec6}

We are now ready to prove our main result, \cref{thm:main}. The reader may find it helpful to revisit the overview given in the introduction.

\subsection{Setting up parameters.} \label{sec61} We begin by defining parameters and scales to be used in the proof. 

Let $r$ be the number of colours; we will fix this for the remainder of the proof and we may assume it is sufficiently large. Let $C_0$ be a suitable large positive integer (independent of $r$), recall that $\eps_0 := \frac{1}{10}$, and set 
\begin{equation}\label{params-defs} K := C_0r^{8}, \quad t := K^2, \quad V = (\lceil r^{4 + \eps_0}\rceil^C)!\quad \mbox{and} \quad N := \exp\exp (r^{50}),\end{equation} where here $C$ is the constant in \cref{main-sec3}.
Define \begin{equation}\label{b-set-def} B_0 := \{  V^{4^{i}} : i = 1,2,\dots, K^2\}.\end{equation}
We now define a doubly-indexed sequence of positive integer scales $(H_{i,j})_{i \in [t], j \in [2K]}$  by
\begin{equation}\label{hij-def}  H_{i,j} := \lfloor \exp \exp (r^{25}(4Ki + j)) \rfloor.\end{equation}
Note that we have the crude bounds
\begin{equation}\label{h-heir} \exp\exp((\log \log N)^{1/10}) < \max B_0 < H_{1,1} < \cdots < H_{1,2K} < H_{2,1} < \cdots < H_{t, 2K} <  e^{(\log N)^{1/10}},\end{equation} provided $r$ is large enough. We will also use the auxiliary scales $H_{i,0}$ defined by the same formula \cref{hij-def}. For $i \in [t]$ and $j \in [K]$, define $\mc{P}_{i,j}$ to be the set of primes satisfying $H_{i,2j-1} \le p \le H_{i,2j}$. We note that with this choice of parameters we have, by Mertens' theorem, $\sum_{p \in \mc{P}_{i,j}} \frac{1}{p} \gg r^{25}$.

\subsection{Positivity for $x, xy$}\label{sec:positivity}
The first step of the proof is to isolate the colour class in which we will eventually find our configuration $\{x + y, xy\}$, and to show that it is rich in configurations $\{x , xy\}$. This is a mild variant of \cite[Theorem~3.6]{Ric25}, which itself is related to results of Ahlswede, Khachatrian and S{\'{a}}rk{\"{o}}zy \cite{AKS99} and Davenport and Erd\H{o}s \cite{DE36}.

Consider an $r$-colouring $A_1 \cup \cdots \cup A_r = [N]$. For each $b \in B_0$ we have 
\[ \mb{E}_{n\in [N]}^{\log}\sum_{j = 1}^{r}\mbf{1}_{A_{j}}(b n) = \mb{E}_{n \in [N]}^{\log} 1_{[N]}(b n) = \frac{H_{N/b}}{H_N} \ge \tfrac{1}{2},\] where here $H_N$ denotes the harmonic sum. The last bound here follows (comfortably) using \cref{h-heir}. By summing over all $b \in B_0$ and an appeal to the pigeonhole principle, there is some colour class $A = A_{j}$ such that 
\[ \sum_{b \in B_0} \mb{E}_{n \in [N]}^{\log} 1_A(b n) \ge K^2/2r,\] which implies that $\mb{E}_{n \in [N]}^{\log} 1_A(b n) \ge 1/4r$ for at least $K^2/4r \ge K$ elements $b \in B_0$. Fix a set $B \subset B_0$ of $K$ such elements. We fix the colour class $A$ for the remainder of the proof.

 By repeated applications of \cref{lem:Elliot}, we have
\[ \mb{E}^{\log}_{n \in [N], p_{i,1} \in \mc{P}_{i,1},\dots , p_{i,j } \in \mc{P}_{i,j}} 1_A(b p_{i,1} \cdots p_{i,j} n)  \ge 1/8r\] for any $i \in [t]$, any $j \le K$ and for any $b \in B$. Note here that the error term arising from this repeated application of \cref{lem:Elliot} is dominated by $\ll K \max_{i,j}\big( \sum_{p \in \mc{P}_{i,j} }\frac{1}{p}\big)^{-1/2} \ll K r^{-25/2} \ll r^{-3}$.

Let the elements of $B$ be $b_1 < \cdots < b_K$. Then, applying the above with $b = b_{j}$ and summing over $1 \le j \le K$, we obtain
\[\sum_{j = 1}^{K}\mb{E}^{\log}_{n\in [N], p_{i,1}\in \mc{P}_{i,1}, \dots, p_{i,K}\in \mc{P}_{i,K}}1_A(b_{j}p_{i,1}\cdots p_{i,j}n) \ge \frac{K}{8r}.\] (Note here that, for the term with index $j$, we can include the extra averages over $\mc{P}_{i,j+1},\dots \mc{P}_{i,K}$ with no change to the expression.)
By Cauchy--Schwarz it follows that
\[\mb{E}^{\log}_{n\in [N],p_{i,1}\in \mc{P}_{i,1},\dots, p_{i,K}\in \mc{P}_{i,K}}\sum_{1\le j, j'\le K}1_A(b_{j}p_{i,1}\cdots p_{i,j}n)1_A(b_{j'}p_{i,1}\cdots p_{i,j'}n) \ge 2^{-6}\big(\frac{K}{r}\big)^2.\]
Since $K = r^{8}$, if $r$ is large enough we may exclude the $O(K)$ pairs of indices with $|j - j'|\le 1$ at the loss of at most a factor $2$. By symmetry we are also free to only include the pairs with $j > j'$ (at the loss of another factor of 2), and we thereby obtain
\begin{equation}\label{eq551}\mb{E}^{\log}_{n\in [N],p_{i,1}\in \mc{P}_{i,1},\dots, p_{i,K}\in \mc{P}_{i,K}}\sum_{\substack{1\le j' < j \le K\\ j \ge j' + 2}}1_A(b_{k}p_{i,1}\cdots p_{i,j'}n)1_A(b_{j}p_{i,1}\cdots p_{i,j}n) \ge 2^{-8}\big(\frac{K}{r}\big)^2.\end{equation}
By another repeated application of \cref{lem:Elliot} we have
\begin{align*} & \mb{E}^{\log}_{n\in [N],p_{i,1}\in \mc{P}_{i,1},\dots, p_{i,K}\in \mc{P}_{i,K}}1_A(b_{j'}p_{i,1}\cdots p_{i,j'}n)1_A(b_{j}p_{i,1}\cdots p_{i,j}n) \\ & \qquad\qquad  = \mb{E}^{\log}_{n\in [N],p_{i,1}\in \mc{P}_{i,1},\dots, p_{i,K}\in \mc{P}_{i,K}} 1_A(b_{j'} n)1_A(b_{j}p_{i,j'+1}\cdots p_{i,j}n)  + O(r^{-3})\end{align*} for each pair $j, j'$ with $j > j'$. From this and \cref{eq551}, it follows (again assuming $r$ large enough) that 
\[\mb{E}^{\log}_{n\in [N],p_{i,1}\in \mc{P}_{i,1},\dots, p_{i,K}\in \mc{P}_{i,K}} \sum_{\substack{1\le j' < j\le K\\ j \ge  j' +  2}}1_A(b_{j'}n)1_A(b_{j}p_{i,j' + 1}\cdots p_{i,j}n) \ge 2^{-9} \big(\frac{K}{r}\big)^2.\]
Recall that this is true for all $i \in [t]$. By pigeonhole, for each $i$ there is some $j'(i)$ such that 
\[ \mb{E}^{\log}_{n\in [N],p_{i,1}\in \mc{P}_{i,1},\dots, p_{i,K}\in \mc{P}_{i,K}}\sum_{j =  j'(i) +  2}^K1_A(b_{j'(i)}n)1_A(b_{j}p_{i,j'(i) + 1}\cdots p_{i,j}n) \ge 2^{-9} \frac{K}{r^2}.\]
Pass to a subset $I \subset [t]$ of size at least $t/K$ such that $j'(i)$ does not depend on $i \in I$, and denote by $j'$ the common value of these $j'(i)$. Writing $b := b_{j'}$ and $f(n) := 1_A(b n)$, we then have
\begin{equation}\label{main-ramsey-statement}\mb{E}^{\log}_{n\in [N],p_{i,1}\in \mc{P}_{i,1},\dots, p_{i,K}\in \mc{P}_{i,K}} \sum_{j = j' + 2 }^Kf(n)1_A(b_j p_{i,j' + 1}\cdots p_{i,j}n) \ge 2^{-9} \frac{K}{r^2}\end{equation} for all $i \in I$. Fix this choice of $j'$ (and hence of $b = b_{j'}$ and the function $f$) for the rest of the proof. Define also $I_* := I \setminus \{ \max I\}$ to be the elements of $I$ except the largest one; thus $|I_*| \ge |I|/2$.

\subsection{Proof of the main theorem}  We think of pairs $(i,j)$ (with $i \in I_*$ and $j \ge j' + 2$) as `scales' in the proof. Associated to any scale will be a pair of `projection' operators in the sense of \cref{proj-def}. Define $Q_{j} := b_{j}/b^2 V$. Note that $Q_{j}$ is an integer (in fact it equals $V^{4^{j} - 2\cdot 4^{j'} + 1}$). 

For each pair $(i, j)$ there will be two important projection operators $\Pi$, namely 
\begin{equation}\label{pi-pm} \Pi^{\sml}_{i,j} := \Pi_{Q_{j-1}, H_{i_+,0}} \quad \mbox{and} \quad \Pi^{\lrg}_{i, j} := \Pi_{Q_{j}, H_{i,0}}.\end{equation}
Here, $i_+$ denotes the next largest element in $I$ after $i$, which exists since $i \in I_* = I \setminus \{ \max I\}$. We informally refer to these as the  `small' and `large' projections associated to $(i,j)$.

We first apply the small projection operator to \cref{main-ramsey-statement} using \cref{maximal}, or more accurately \cref{avg-max}. Taking $\eta = 2^{-10} r^{-2}$ there, we have
\begin{equation}
\mb{E}^{\log}_{n\in [N],p_{i,*}\in \mc{P}_{i,*}} \sum_{j = j'+2}^K\Pi_{i,j}^{\sml}f(n)  1_A(b_{j}p_{i,k + 1}\cdots p_{i,j}n) \\  \ge \frac{\eta}{8} (2^{-9} \frac{K}{r^2}) - \frac{\eta^2}{8} K \gg  \frac{K}{r^4}.\label{proj-ramsey}
\end{equation} Here, and below, $\mb{E}^{\log}_{p_{i,*} \in \mc{P}_{i,*}}$ is shorthand for $\mb{E}^{\log}_{p_{i,1} \in \mc{P}_{i,1},\dots, p_{i,K} \in \mc{P}_{i,K}}$.
Now observe that by \cref{almost-per-proj} we have
\begin{align}\nonumber \Pi_{i,j}^{\sml}f(n) & =   \Pi_{i,j}^{\sml}f\big(n + \frac{b_{j}}{b^2 }p_{i,j' + 1}\cdots p_{i,j}\big) + O\Big(\frac{b_{j}}{b^2} \frac{p_{i,j'+1} \cdots p_{i,j}}{H_{i_+,0}}\Big)\\ & = \Pi_{i,j}^{\sml}f\big(n + \frac{b_{j}}{b^2 }p_{i,j' + 1}\cdots p_{i,j}\big) + O(r^{-10}). \label{eq498}\end{align}
The key points to observe here in applying \cref{almost-per-proj} are that 
$Q_{j - 1}  = \frac{b_{j - 1}}{b^2} V \mid \frac{b_{j}}{b^2}$ by the definitions of the $b_j$s, and also 
\[ \frac{b_{j}}{b^2 }p_{i,j' + 1}\cdots p_{i,j} \le V^{4^{K^2}} \prod_{j=1}^K H_{i,2j} < V^{4^{K^2}} H_{i,2K}^2 < r^{-10} H_{i_+,0}.\] The inequalities here are all very comfortably true (when $r$ is large); we have $r^{10} < V^{4^{K^2}} < H_{1,1} < H_{i,2K}$, that $H_{i,2j}^2 < H_{i, 2(j+1)}$ for all $j$, and that $H_{i,2K}^4 < H_{i_+,0}$, all of which follow using \cref{h-heir}.
From \cref{proj-ramsey,eq498} we have
\[
 \sum_{j = j'+2}^K\mb{E}^{\log}_{n\in [N],p_{i,*}\in \mc{P}_{i,*}}\Pi_{i,j}^{\sml} f\big(n + \frac{b_{j}}{b^2 }p_{i,j' + 1}\cdots p_{i,j}\big) 1_A(b_{j}p_{i,j' + 1}\cdots p_{i,j}n)  \gg \frac{K}{r^4}.\]
This, recall, is for all $i \in I_*$. Summing over all these $i$ gives
\begin{equation}\label{proj-ramsey-44}
 \sum_{i \in I_*} \sum_{j = j'+2}^K\mb{E}^{\log}_{n\in [N],p_{i,*}\in \mc{P}_{i,*}}\Pi_{i,j}^{\sml}f\big(n + \frac{b_{j}}{b^2}p_{i,j' + 1}\cdots p_{i,j}\big)1_A(b_{j}p_{i,j' + 1}\cdots p_{i,j}n) \gg \frac{K|I|}{r^4}.
\end{equation}
Suppose we had a similar result with $\Pi_{i,j}^{\sml}f$ replaced by $f$, that is
\begin{equation}\label{desired-result}
\sum_{i \in I_*}\sum_{j = j'+2}^K\mb{E}^{\log}_{n\in [N],p_{i,*}\in \mc{P}_{i,*}} f\big(n + \frac{b_{j}}{b^2 }p_{i,j' + 1}\cdots p_{i,j}\big)1_A(b_{j}p_{i,j' + 1}\cdots p_{i,j}n) \gg \frac{K|I|}{r^4}.
\end{equation}
In particular, for some choice of $i, j, p_{i,j'+1},\dots, p_{i,j}$ and $n \ge 3$ we would then have
\[ f\big(n + \frac{b_{j}}{b^2 }p_{i,j' + 1}\cdots p_{i,j}\big)1_A(b_{j}p_{i,j' + 1}\cdots p_{i,j}n) > 0.\]
Taking $x := b n$ and $y := \frac{b_{j}}{b}p_{i,j' + 1}\cdots p_{i,j}$ (and recalling that $f(n) = 1_A(b n)$) we then have $x  + y, xy \in A$, and the proof is complete. 

It remains to prove that we do indeed have \cref{desired-result}. As described in the introduction, we deduce it from \cref{proj-ramsey-44} in two steps. First, we replace the `small' projections $\Pi_{i,j}^{\sml}f$ in \cref{proj-ramsey-44} by the `large' projections $\Pi_{i,j}^{\lrg}f$. The error in making this replacement is
\begin{equation}\label{eq457-pre} \sum_{i \in I_*} \sum_{j = j'+2}^K\mb{E}^{\log}_{n\in [N],p_{i,*}\in \mc{P}_{i,*}} \big( \Pi_{i,j}^{\sml}f - \Pi_{i,j}^{\lrg} f\big)\big(n + \frac{b_{j}}{b^2 }p_{i,j' + 1}\cdots p_{i,j}\big)1_A(b_{j}p_{i,j' + 1}\cdots p_{i,j}n).\end{equation} By \cref{log-avg-shift} and the crude bounds $b_{j} \le V^{4^{K^2}}$, $p_{i,*} \le H_{t, 2K}$ this is
\begin{equation}\label{eq457} \sum_{i \in I_*} \sum_{j \ge j'+2 }\mb{E}^{\log}_{p_{i,*}\in \mc{P}_{i,*}} \mb{E}^{\log}_{n \in [N]} \big( \Pi_{i,j}^{\sml}f - \Pi_{i,j}^{\lrg} f\big)(n)\psi_{i,j,p_{i,*}}(n)  + O\Big(|I| K\frac{\log (V^{4^{K^2}}H^{2K}_{t, 2K})}{\log N} \Big),\end{equation}
where
\[ \psi_{i,j,p_{i,*}}(n) := 1_A\big(b_{j}p_{i,j' + 1}\cdots p_{i,j}( n -  \frac{b_{j}}{b^2} p_{i,j'+1} \cdots p_{i,j})\big).\]
For the rest of the proof (as in \cref{lem:tel-help}) we use the notation $\langle g_1, g_2 \rangle := \mb{E}_{n \in [N]}^{\log}g_1(n) \overline{g_2(n)}$ and $\Vert g \Vert^2 := \langle g, g\rangle = \mb{E}_{n \in [N]}^{\log} |g(n)|^2$. Using \cref{params-defs,h-heir}, the error term in \cref{eq457} is seen to be $O(|I|K r^{-10})$. Thus \cref{eq457} is
\[ \sum_{i \in I} \sum_{j = j'+2}^K\mb{E}^{\log}_{p_{i,*}\in \mc{P}_{i,*}} \langle \Pi_{i,j}^{\sml}f - \Pi_{i,j}^{\lrg} f,\psi_{i,j,p_{i,*}}\rangle  + O(|I|K r^{-10}).\]
By Cauchy--Schwarz and the $1$-boundedness of the functions $\psi$, this is bounded above by
\begin{align}\nonumber  \sum_{i \in I_*} \sum_{j = j'+2}^K & \Vert \Pi_{i,j}^{\lrg}f - \Pi_{i,j}^{\sml}f  \Vert +  O(|I|K r^{-10}) \\ & \le (|I| K)^{1/2}\Big( \sum_{i \in I_*} \sum_{j = j'+2}^K \Vert \Pi_{i,j}^{\lrg}f - \Pi_{i,j}^{\sml}f  \Vert^2\Big)^{1/2}  + O(|I| K r^{-10}). \label{eq712aa}\end{align}
For each $i,j$ we apply \cref{lem:tel-help} with $q = Q_{j-1}$, $q' = Q_j$, $H = H_{i_+,0}$ and $H' = H_{i,0}$, obtaining
\begin{align}\nonumber \Vert \Pi_{i,j}^{\lrg}f - \Pi_{i,j}^{\sml}f  \Vert^2  & \le \Vert \Pi_{i,j}^{\lrg}f\Vert^2  - \Vert \Pi_{i,j}^{\sml}f  \Vert^2 + O\Big(\frac{\log Q_{j} H_{i_+,0}}{\log N}\Big) + O\Big( \frac{Q_{j} H_{i,0}}{Q_{j - 1} H_{i_+,0}}  \Big) \\ & \le \Vert \Pi_{i,j}^{\lrg}f\Vert^2  - \Vert \Pi_{i,j}^{\sml}f  \Vert^2 + r^{-10}. \label{713aa}\end{align}
The explain the last line here, we can bound the first error term by $< (\log N)^{-1/2} < r^{-20}$ using \cref{h-heir}. The second error term can be bounded using $Q_{j} < \max B_0$ and the fact that $H_{i_+, 0} \ge H_{i+1,0} > r^{20} (\max B_0) H_{i,0}$, which can be verified using the definitions \cref{b-set-def,hij-def}. 

Summing \cref{713aa} over $i,j$ gives
\[\sum_{i \in I_*} \sum_{j = j'+2}^K \Vert \Pi_{i,j}^{\lrg}f - \Pi_{i,j}^{\sml}f  \Vert^2 \le  \sum_{i \in I_*}\sum_{j = j'+2}^K \big(  \Vert \Pi_{i,j}^{\lrg}f\Vert^2 - \Vert \Pi_{i,j}^{\sml}f  \Vert^2\big) + O(|I|K r^{-10}).\]
Recalling the definitions \cref{pi-pm} of the two projection operators, we see that the bracketed sum has considerable cancellation; the only uncancelled positive terms are the $\Vert \Pi_{i,j}^{\lrg}f\Vert^2$ terms from scales $(i,j)$ which are not of the form $(\overline{i}_+, \overline{j} - 1)$ for some other scale $(\overline{i}, \overline{j})$, that is to say with $i = \min( I)$ or $j = K$; thus the bracketed sum is bounded by $|I| + K$. It follows that \cref{eq712aa} is bounded by
\[ \le (|I| K)^{1/2} \big(|I| + K + O(r^{-10})\big)^{1/2} + O(|I|K r^{-10}) \ll C_0^{-1/2}r^{-4} |I|K,\] using here that $K = C_0 r^{8}$ and $|I| \ge t/K = K$.

If the constant $C_0$ is chosen large enough, this means that \cref{eq457-pre} is small compared with the RHS of \cref{proj-ramsey-44}.

To summarise so far, we have replaced the `small' projections $\Pi^{\sml}_{i,j}$ in \cref{proj-ramsey-44} by the `larger' ones $\Pi^{\lrg}_{i, j}$ at the loss of only the quality of the implied constant, that is to say we have shown

\[
 \sum_{i \in I_*} \sum_{j = j'+2}^K\mb{E}^{\log}_{n\in [N],p_{i,*}\in \mc{P}_{i,*}}\Pi_{i,j}^{\lrg}f(n + \frac{b_{j}}{b^2 }p_{i,j' + 1}\cdots p_{i,j})1_A(b_{j}p_{i,j' + 1}\cdots p_{i,j}n) \gg \frac{K|I|}{r^4}.
\]

To complete the proof of \cref{desired-result} (and hence of \cref{thm:main}) we now replace the copies of $\Pi_{i,j}^{\lrg} f$ by $f$ itself. For this we can work one value of $(i,j)$ at a time; thus it is enough to show that, for each $(i,j)$, 

\begin{equation}\label{proj-ramsey-46}
 \mb{E}^{\log}_{n\in [N],p_{i,*}\in \mc{P}_{i,*}}\big( f - \Pi_{i,j}^{\lrg}f\big)(n + \frac{b_{j}}{b^2 }p_{i,j' + 1}\cdots p_{i,j})1_A(b_{j}p_{i,j' + 1}\cdots p_{i,j}n) \le r^{-4-\eps_0}.
\end{equation}
(Here $\eps_0 = \frac{1}{10}$ again). To prove this we use \cref{main-sec3}. Indeed, we note that the LHS of \cref{proj-ramsey-46} is of the form 
\[  \mb{E}_{n \in [N],p \in \mc{P},p' \in \mc{P}'}^{\log}  f_1(n + \lambda p p') f_2(\lambda n p p').\]
(which is exactly the expression in \cref{key-assump}) where $f_1 = f - \Pi_{Q_{j},H_{i,0}}f$, $f_2(n) = 1_A(b^2 n)$, $\lambda = b_{j}/b^2$, $\mc{P} = \mc{P}_{i, j}$, $\mc{P}' = \mc{P}_{i, j'+1}\cdots \mc{P}_{i,j-1}$ and $k = j - j' -1 \in \N$.

Note here that every element of $\mc{P}'$ has just one representation in this product since all primes in $\mc{P}_{i, j'+1}$ are much smaller than those in $\mc{P}_{i, j'+2}$, and so on, and so $\mb{E}^{\log}_{p_{i,j'+1}\in \mc{P}_{i,j'+1},\dots, p_{i,j-1}\in \mc{P}_{i,j-1}}$ is the same thing as $\mb{E}^{\log}_{p' \in \mc{P}'}$.

The setup for the application of \cref{main-sec3} requires some discussion. We address the various requirements in the statement of that proposition in turn.
\begin{itemize}
\item The parameter $k$ will be $j - j'-1$. Note $1 \le k \le K$, so the condition $k \le \log \log N$ is satisfied due to the choices \cref{params-defs}.
    \item We will take $\delta := \lceil r^{4 + \eps_0}\rceil^{-1}$ (the aim being to show that the LHS of \cref{proj-ramsey-46} is at most $\delta$). The conditions $1/\delta \le \log \log N$ and $k \le \delta^{-10}$ are then immediately checked.
    \item We take $\mc{P} = \mc{P}_{i,j}$ and $\mc{P'} = \mc{P}_{i,j'+1} \cdots \mc{P}_{i,j-1}$. For notational consistency with \cref{main-sec3}, write $\mc{P}'_{\ell} := \mc{P}_{i,j' + \ell}$ for $\ell \in [k]$. Thus, by definition, $\mc{P}'_{\ell}$ is the set of primes in the interval $I_{\ell} = [H_{i, 2(j' + \ell) - 1}, H_{i, 2(j' + \ell)}]$, which is exactly the situation in \cref{main-sec3}. By \cref{hij-def} and the choice of parameters we have $\log \log (\max(I_{\ell})) - \log \log (\min(I_{\ell})) \ge r^{25} > k\delta^{-4-\eps_0}$. (This is essentially the `pinch point' for the analysis; for the main result to have the stated exponent of 50 we need $(4 + \eps_0)^2 < 17$ here.)
    \item We take $P_1 = H_{i, 2j - 1}$, $P_2 = H_{i, 2j}$. The condition $P_2 < \exp((\log N)^{1/4})$ is implied by \cref{h-heir}, if $C_2$ is large enough.
    \item We take
    $P'_1 = H_{i, 2j'+ 1}$, $P'_2 = H^2_{i, 2j - 2}$. Note here that $\min(\mc{P}') \ge P'_1$ and $\max (\mc{P}') \le H_{i, 2j' + 2}  \cdots H_{i, 2j - 2} \le P'_2$, as required, using here that $H_{i,j}^2 < H_{i,j+1}$. The condition $P'_2 \ge \exp\exp((\log \log N)^{1/10})$ follows immediately from \cref{h-heir}.
  \item The condition $\lambda \le e^{(\log N)^{1/4}}$ follows from \cref{h-heir} and the fact that $\lambda \le \max B_0$. 

  \item That all prime factors of $\lambda$ are less than $P'_1$ is immediate from the lower bound $H_{1,1} > \max B_0$.
\end{itemize}
Suppose that \cref{proj-ramsey-46} does not hold. By the above discussion we are in a position to apply \cref{main-sec3}. Note that $V$ in the conclusion there is, with our choice of parameters, exactly the same as $V$ in \cref{params-defs}. Since $\lfloor P_1^{1/8}\rfloor \ge \lfloor H_{i,1}^{1/8} \rfloor > H_{i,0}^2$, we may take the parameter $H$ in \cref{main-sec3} to be $H_{i,0}^2$. The conclusion of \cref{main-sec3} is then that 
\[ \Vert f - \Pi_{Q_{j}, H_{i,0}} f \Vert_{U^1_{\log}[N; Q_{j}, H_{i,0}^2]} \gg K^{-O(1)}.\] (Here we observed from the various definitions that $Q_j = \lambda V$.) However, this is contrary to \cref{lem:proj-check}, which asserts that the LHS is $\ll H_{i,0}^{-1}$, which is enormously smaller. This contradiction shows that we indeed have \cref{proj-ramsey-46}, and all of the required statements are proven.

\section{Further remarks}

We end the main body of the paper with a series of remarks regarding the bounds obtained for the pattern $\{x+y,xy\}$ and related patterns.

First of all, we comment that there are two different ways in which the double exponential bound in the main theorem seems hard to improve using anything like the methods of this paper. The first is that it seems difficult to avoid the need to define a highly divisible set such as the set $B_0$ in \cref{b-set-def}, and any such definition seems to immediately lead to elements of double exponential size in $r$. Second, the hierarchy of scales \cref{h-heir} needed to be chosen with $\log \log (H_{i, j+1}) - \log \log (H_{i,j}) \gg 1$ in order that the primes in this range satisfy $\sum_{p \in \mc{P}} \frac{1}{p} \gg 1$, which is crucial in the application of \cref{main-sec3}. It is possible to show using arguments somewhat related to those in \cite{Tao24} that one cannot do appreciably better by choosing an alternative set to the primes. In particular, when applying \cref{lem:Elliot} with an alternate set of integers $\mc{P}$, the error term is dominated by $\gamma(\mc{P})^{1/2}$ and one can prove that for any set $\mc{P}\subseteq [2,X]$ one has $\gamma(\mc{P})\gg (\log\log X)^{-1}$. 

Next we make some comments on the potential for extending the underlying analytic method to handle the pattern $\{x,x+y,xy\}$ (for which partition regularity was established by Moreira \cite{Mor17}, but with essentially no bounds). Presumably any such approach would require one to (at least) establish an inverse theorem establishing some structure assuming that 
\begin{equation}\label{3-fns-assump} \big| \mb{E}_{n \in [N], p \in \mc{P}}^{\log} f_1(n) f_2(n + p) f_3(np)\big| \ge \delta,\end{equation} where $\mc{P}$ is a suitable set of almost primes (compare here with \cref{key-assump}. The following two rather different examples suggest this may be far from straightforward.

\begin{itemize}
    \item Suppose first that $f_2 = 1$. Let $(\xi_p)_{p \in \mc{P}}$ be an arbitrary sequence of unit complex numbers, and define $f_1(n) := \xi_p$ if $p$ is the least prime in $\mc{P}$ which divides $n$, and $f_1(n) = 0$ otherwise. Set $f_3(n) := \overline{f_1(n)}$. Assuming that $\sum_{p \in \mc{P}} \frac{1}{p} \ggg 1$, the (logarithmic) proportion of $n$ for which $f_1(n) = 0$ is negligible. Now observe that $f_1(n) f_3(pn) = 1$ if the least prime factor of $n$ in $\mc{P}$ is less than $p$. On average over $p, n$, one expects this to happen half the time. If, one other other hand, the least prime factor of $n$ is $p' > p$ then we have $f_1(n) f_3(pn) = \xi_{p'}\overline{\xi_p}$, and typically we expect cancellation of this when summed over $p, p'$. Examples of this type therefore give \cref{3-fns-assump} with $\delta \approx 1/2$, but with $f_1, f_3$ only having rather weak structure.
    \item Now suppose that $f_1(n) = e(\alpha n^2)$, $f_2(n) = e(-\alpha n^2)$ and $f_3(n) = e(2\alpha n)$ for some $\alpha \in \R$. One may then observe that $f_1(n)f_2(n + p)f_3(np) = e(-\alpha p^2)$. If $P$ is the scale of $\mc{P}$ then this is $\approx 1$ for $|\alpha| \lessapprox P^{-2}$. 
\end{itemize}
Even with an inverse theorem for \cref{3-fns-assump} in hand, it is far from clear how the other arguments of the paper might be modified.

\appendix  
\section{Properties of averages}
In this appendix we assemble simple properties of (mostly) logarithmic averages. Throughout the appendix we assume $N \ge 2$ to avoid trivialities. For $m \in \R_{\ge 1}$, $H_m$ denotes the harmonic sum $\sum_{n \le m} \frac{1}{n}$; we do not require $m$ to be an integer. The first lemma concerns the behaviour of averages (both uniform and logarithmic) under shifts.

\begin{lemma}\label{avg-shifts}
Let $f : \N \rightarrow \C$ be a 1-bounded function and let $h \in \Z$. Then 
\begin{equation}\label{ord-avg-shift} \big| \mb{E}_{n \in [N]} f(n) - \mb{E}_{n \in [N]} f(n + h)\big| \ll \frac{|h|}{N} \end{equation} and, if $h \ne 0$,
\begin{equation}\label{log-avg-shift} \big| \mb{E}^{\log}_{n \in [N]} f(n) - \mb{E}^{\log}_{n \in [N]} f(n + h)\big| \ll \frac{1 + \log |h|}{\log N}. \end{equation}
\end{lemma}
\begin{proof} \cref{ord-avg-shift} is straightforward. For \cref{log-avg-shift}, we may suppose $|h| \le N/2$ else the result is trivial. 
Without loss of generality we may suppose $h$ is positive, since the case $h$ negative follows from the positive case. We have
\begin{equation}\label{exp-a1} \sum_{n \in [N]} \frac{f(n + h)}{n} - \sum_{n \in [N]}\frac{f(n)}{n} = \sum_{m = h + 1}^N \big( \frac{f(m)}{m - h} - \frac{f(m)}{m}\big) - \sum_{n = 1}^h \frac{f(n)}{n} + \sum_{n = N - h +1}^{N} \frac{f(n + h)}{n}.\end{equation}
The second sum on the right is $\ll 1 + \log h$, whilst the third is $\le \log N - \log (N - h + 1) + O(1) \ll 1$ since $h \le N/2$. Finally, the first sum on the right is bounded above by $h \sum_{m = h+1}^N \frac{1}{m (m - h)}$. Since $\sum_{m = h+1}^{2h} \frac{1}{m (m - h)} \ll \frac{1}{h}\sum_{m = h+1}^{2h} \frac{1}{m - h} \ll \frac{1+ \log h}{h}$, and $\sum_{m = 2h + 1}^N \frac{1}{m (m - h)} \ll \sum_{m > 2h} m^{-2} \ll h^{-1}$, the first sum on the right in \cref{exp-a1} is bounded by $\ll 1 + \log h$. Putting all this together, the result follows.
\end{proof}

Next we give a result about splitting into residue classes.
\begin{lemma}\label{residue-split} Let $f : \Z \rightarrow \C$ be $1$-bounded. Let $q \in \N$. Then
\[  \mb{E}_{a \in \{0,1,\dots, q-1\}} \mb{E}_{n \in [N]}^{\log} f(qn + a) = \mb{E}_{n \in [N]}^{\log} f(n) + O\Big(\frac{1 + \log q}{\log N}\Big).\]
\end{lemma}
\begin{proof} We may suppose $2 \le q \le N$ since the result is trivial otherwise.
The LHS may be expanded as 
\[ \frac{1}{H_N} \sum_{a \in\{0,1,\dots, q-1\} }\sum_{n \in [N]} \frac{f(qn + a)}{qn}.\]
The change if we replace $qn$ in the denominator by $qn + a$ is bounded above by
   \[ \ll \frac{1}{\log N}\sup_a \sum_{n \in [N]} \Big| \frac{1}{n} - \frac{1}{n + a/q}\Big| \ll \frac{1}{\log N},\] which is acceptable.
   If we make this change, the resulting expression is
   \[ \frac{1}{H_N} \sum_{q \le n' \le qN + (q-1)} \frac{f(n')}{n'} = \mb{E}_{n \in [N]}^{\log} f(n) + O\Big(\frac{H_q}{H_N}\Big) + O\Big(\frac{H_{qN + (q - 1)} - H_N}{H_N}\Big).\] The two error terms are $\ll \frac{1 + \log q}{\log N}$, and this concludes the proof.
\end{proof}

We also need the following related result.
\begin{lemma}\label{frobenius-coin}
Let $q, b$ be coprime positive integers and let $H$ be a further positive integer parameter. Let $f : \N \rightarrow \C$ be a 1-bounded function. Then
\[ \mb{E}_{n \in [N]}^{\log} \mb{E}_{h \in [H]} f(qn + bh) = \mb{E}^{\log}_{n \in [N]} f(n) + O\Big(\frac{1 + \log q + \log b H}{\log N}\Big) + O\big(\frac{q}{H}\big).\]
\end{lemma}
\begin{proof}
Clearly we may assume that $H \ge q$, as the result is trivial otherwise. If we replace $H$ by $\tilde H := q \lfloor H/q\rfloor$, the LHS changes by at most $O(q/H)$. It therefore suffices to consider the case $q \mid H$. In this case we establish the result without the $O(q/H)$ error term.  We have $f(qn + bh) = f(q (n + \sigma_h) + (bh)_q)$, where $(bh)_q$ denotes the unique element of $\{0,1,\dots, q-1\}$ congruent to $bh \md{q}$, and $\sigma_h := \frac{1}{q}(bh - (bh)_q)$. By \cref{log-avg-shift}, we have 
\[ \mb{E}^{\log}_{n \in [N]} f(q (n + \sigma_h) + (bh)_q) = \mb{E}^{\log}_{n \in [N]} f(q n + (bh)_q) + O\Big( \frac{1 + \log bH}{\log N} \Big).\]
However, since $(bh)_q$ ranges over $\{0,1,\dots, q - 1\}$ as $h$ ranges over any interval of length $H/q$,
\[ \mb{E}_{h \in [H]} \mb{E}^{\log}_{n \in [N]} f(q n + (bh)_q) = \mb{E}_{a \in \{0,1,\dots, q-1\}} \mb{E}^{\log}_{n \in [N]} f(q n + a).\]
The result now follows from \cref{residue-split}. 
\end{proof}

The next result states that logarithmic averages are essentially preserved under dilations. This is standard and appears, for instance, as \cite[Lemma~2.1]{Ric25}.
\begin{lemma}\label{lem:dilate}
Let $f:\N \to \C$ be $1$-bounded and let $q \in \N$. Then 
\[\Big|\mb{E}_{n\in [N]}^{\log}\big( f(n) - q\mbf{1}_{q|n}f(n/q)\big)\Big| \ll \frac{\log q}{\log N}.\]
\end{lemma}
\begin{proof}
When $q = 1$ the result is trivial, so suppose $q \ge 2$. By definition,
\[\mb{E}_{n\in [N]}^{\log}\big( f(n) - q\mbf{1}_{q|n}f(n/q)\big)  = \frac{1}{H_N}\sum_{n\in [N]}\frac{f(n) - q\mbf{1}_{q|n}f(n/q)}{n} =  \frac{1}{H_N} \Big( \sum_{n \in [N]} \frac{f(n)}{n} - \sum_{n' \in [N/q]}\frac{f(n')}{n'}\Big) . \]
This is bounded by $\le \frac{1}{H_N}\big( H_{N} - H_{N/q}\big) = \frac{1}{H_N} (\log q + O(1))$, and the result follows.
\end{proof}

We next require the logarithmic version of Elliott's inequality. The proof is exactly that given in \cite[Corollary~2.3]{Ric25} modulo tracking error terms. 
\begin{lemma}\label{lem:Elliot}
Let $\mc{P}$ be a finite set of primes, all bounded by $P$. Let $f:\N \to \C$ be $1$-bounded. We have that 
\[\Big|\mb{E}_{n\in [N]}^{\log}f(n) - \mb{E}_{n\in [N],p\in \mc{P}}^{\log}f(pn)\Big| \ll \frac{\log P}{\log N} + \Big(\sum_{p\in \mc{P}}\frac{1}{p}\Big)^{-1/2}.\]
\end{lemma}
\begin{proof}
By \cref{lem:dilate} applied with $\tilde f(n) := f(pn)$, for each $p \in \mc{P}$ we have 
\[ \mb{E}_{n\in [N]}^{\log}f(pn) = p\mb{E}^{\log}_{n \in [N]}\mbf{1}_{p|n}f(n) + O\big( \frac{\log P}{\log N}\big).\] Therefore by Cauchy--Schwarz we have
\begin{align*}
\Big|\mb{E}_{n\in [N]}^{\log}f(n) - \mb{E}_{n\in [N],p\in \mc{P}}^{\log}f(pn)\Big|  & = \Big|\mb{E}_{n\in [N]}^{\log}\mb{E}_{p\in \mc{P}}^{\log}f(n)(p\mbf{1}_{p|n}-1)\Big| + O\Big(\frac{\log P}{\log N}\Big)\\
&\le \Big(\mb{E}_{n\in [N]}^{\log}\big|\mb{E}_{p\in \mc{P}}^{\log}(p\mbf{1}_{p|n}-1)\big|^2\Big)^{1/2} + O\Big(\frac{\log P}{\log N}\Big).
\end{align*}
By \cite[Proposition~2.2]{Ric25}, we have that 
\begin{align*}
\mb{E}_{n\in [N]}^{\log}&\big|\mb{E}_{p\in \mc{P}}^{\log}(p\mbf{1}_{p|n}-1)\big|^2 \le 9\Big(\sum_{p\in \mc{P}}\frac{1}{p}\Big)^{-1},
\end{align*}
and the result follows.
\end{proof}

 We end with a proposition regarding the behaviour of $U^1_{\log}[N;q,H]$ under replacing $q$ by a multiple or shrinking the interval $H$.
\begin{lemma}\label{gp-compar}
Suppose that $q \mid \tilde q$ and that $\tilde H \tilde q < H q<  N/2$. Then 
\[  \Vert f \Vert_{U^1_{\log}[N; q,H]} \le  \Vert f \Vert_{U^1_{\log}[N; \tilde q,\tilde H]}  + O\Big(\frac{\log |H q|}{\log N}\Big) + O\Big(\frac{\tilde H \tilde q}{Hq}\Big).\]
\end{lemma}
\begin{proof}
First observe that 
\[ \mb{E}_{h \in [H]} f(n + hq) = \mb{E}_{h \in [H],\tilde h \in [\tilde H]} f(n + hq + \tilde h \tilde q) + O\Big(\frac{\tilde H \tilde q}{Hq}\Big)\] by the assumptions and \cref{ord-avg-shift}. Substituting into the definition of $\Vert f \Vert_{U^1_{\log}[N; q,H]}$, we have
\[ \Vert f \Vert_{U^1_{\log}[N; q,H]}^2 = \mb{E}_{n \in [N]}^{\log}\big| \mb{E}_{h \in [H],\tilde h \in [\tilde H]} f(n + hq + \tilde h \tilde q)\big|^2  + O\Big(\frac{\tilde H \tilde q}{Hq}\Big). \]
By Cauchy--Schwarz,
\[ \Vert f \Vert_{U^1_{\log}[N; q,H]}^2 \le \mb{E}_{n \in [N]}^{\log}\mb{E}_{h \in [H]} \big| \mb{E}_{\tilde h \in [\tilde H]} f(n + hq + \tilde h \tilde q)\big|^2  + O\Big(\frac{\tilde H \tilde q}{Hq}\Big). \]
However by \cref{log-avg-shift}, for each $h$ we have
\[  \mb{E}_{n \in [N]}^{\log} \big| \mb{E}_{\tilde h \in [\tilde H]} f(n + hq + \tilde h \tilde q)\big|^2  =  \mb{E}_{n \in [N]}^{\log} \big| \mb{E}_{\tilde h \in [\tilde H]} f(n + \tilde h \tilde q)\big|^2  + O\Big(\frac{\log |hq|}{\log N}\Big). \] Averaging over $h \in [H]$ gives the result.
\end{proof}

\section{An exponential sum estimate over the primes}\label{appC}

In this appendix we prove a log-free exponential sum estimate for the von Mangoldt function with polynomial phase.
\begin{lemma}\label{log-free-weyl}
Let $m \in \N$ and $\eps \in (0,\frac{1}{2})$. Suppose that 
\begin{equation}\label{large-exp-assump} \big| \sum_{n \le X} \Lambda(n) e(n^m \theta) \big| \ge \eps X.\end{equation}
Then there is some $q \in \N$ such that 
\begin{equation}\label{log-free-0} q \le \eps^{-O_m(1)} \quad \mbox{and} \quad \Vert \theta q \Vert_{\R/\Z} \le \eps^{O_m(1)} X^{-m}.\end{equation}
\end{lemma}

\begin{proof}
We proceed via the weaker result with \cref{log-free-0} replaced by

\begin{equation}\label{with-logs} q \le \big(\frac{\log X}{\eps}\big)^{O_m(1)} \quad \mbox{and} \quad \Vert \theta q \Vert_{\R/\Z} \le \big(\frac{\log X}{\eps}\big)^{O_m(1)} X^{-m}.\end{equation}
This is a standard application of the method of Type I/II sums. However, some sources in the literature such as \cite{Har81} lose factors of $X^{o(1)}$ instead of a power of $\log X$ via an invocation of the divisor bound in the proof of Weyl's inequality. This loss can be avoided with a little care, but it is hard to find a convenient source in the literature. One may find an essentially equivalent argument (with the polynomial phase $e(n^m\theta)$ replaced by a general nilsequence) in \cite{GT-mobius}. The key point is that \cite[Proposition 3.1]{GT-mobius} holds verbatim if the M\"obius function $\mu$ is replaced by $\Lambda$. This is established in a standard fashion as in the proof of \cite[Proposition 3.1]{GT-mobius} (which is outsourced to \cite[Section 4]{GT08}, which itself is derivative of standard expositions such as \cite[Chapter 13]{IK-book}) by using Vaughan's identity for $\Lambda$ rather than the variant for $\mu$. One may now run the arguments of \cite[Section 3]{GT-mobius}; in this context most of the language of nilmanifolds is redundant since $e(n^m\theta)$ is a nilsequence on the abelian torus $\R/\Z$. In particular the `complexity' parameter $Q$ is simply $O(1)$. The conclusion of \cite[Section 3]{GT-mobius} is then that, starting from \cref{large-exp-assump}, and setting $\delta := \eps/\log X$, there is some $q \ll_m \delta^{-O_m(1)}$ such that we have $\Vert \theta q\Vert_{\R/\Z} \ll_m \delta^{-O_m(1)}X^{-m}$; this is exactly \cref{with-logs} (noting here that the $\ll$ can be upgraded to $\le$ at the expense of worsening exponents since $\delta \lll 1$).

If $\eps \ge (\log X)^{-1}$ then \cref{with-logs} immediately implies \cref{log-free-0} (after adjusting the exponents $O_m(1)$). To complete the proof of \cref{log-free-weyl}, it therefore suffices to handle the case $\eps \le (\log X)^{-1}$. In this case, from \cref{with-logs} we certainly have $q \le (\log X)^{O_m(1)}$ and $\Vert \theta q \Vert_{\R/\Z} \le (\log X)^{O_m(1)} X^{-m}$. In this case one can obtain an asymptotic for the exponential sum $\sum_{n \le X} \Lambda(n) e(n^m \theta)$ using the Siegel-Walfisz theorem on the distribution of $\Lambda$ in progressions $\md{q}$. These arguments are carried out in detail in work of Hua \cite{Hua38}. Summarising briefly, the main term of this asymptotic at $\theta = \frac{a}{q} + \eta$ will be $\frac{X}{\phi(q)} S(a,q) \nu(\eta X^m)$ where $\nu(y) = \int^1_0 e(y x^m) dx$ satisfies an appropriate van der Corput estimate and $S(a,q) := \sum_{b \in (\Z/q\Z)^*} e(ab^m/q)$ satisfies $|S(a,q)| \ll_{m} q^{1/2 + o_{m}(1)}$. The assumption \cref{large-exp-assump} therefore forces both $\eta X^m \ll \eps^{-1}$ and $q \ll_{m} \eps^{-2 - o_{m}(1)}$. 
\end{proof}

Finally we give the case $k = 1$ of (a slight generalisation of) \cref{lem:crit-estimate}, which was used in the proof of the case $k \ge 2$ of that result. This can be quickly deduced from \cref{log-free-weyl} as a consequence of partial summation.
\begin{lemma}\label{log-free-weyl-2}
Let $m \in \N$ and $\delta,\eta \in (0,\frac{1}{2})$. Suppose that 
\begin{equation}\label{large-exp-assump-2} \big| \sum_{X\le p < (1+\eta)X} e(p^m \theta) \big| \ge \frac{\delta\eta  X}{\log X}.\end{equation}
Then there is some $q \in \N$ such that $q \le (\eta \delta)^{-O_m(1)}$ and $\Vert \theta q \Vert_{\R/\Z} \le (\eta\delta)^{-O_m(1)} X^{-m}$.
\end{lemma}
\begin{proof}
The result is trivial if $\delta \eta \le X^{-1/10}$ (say), so suppose this is not the case. We may replace the assumption \cref{large-exp-assump-2} by
\[ \Big| \sum_{X\le n < (1+\eta)X} \frac{\Lambda(n)}{\log n} e(n^m \theta)\Big| \ge \frac{\delta\eta  X}{2\log X}.\] (The loss of a further factor of 2 here comes from the essentially negligible contribution of the prime power support of $\Lambda$).
Now for $n \ge X$ we have $\frac{1}{\log n} = \frac{1}{\log X} - \int^n_X \frac{dt}{t (\log t)^2}$. Substituting in and applying the triangle inequality gives
\[ \frac{1}{\log X}\Big| \sum_{X \le n < (1+\eta)X} \Lambda(n) e(n^m \theta)\Big| + \int^{2X}_{X} \frac{dt}{t(\log t)^2} \Big| \sum_{t \le n \le (1 + \eta) X} \Lambda(n) e(n^m \theta)\Big| \ge  \frac{\delta\eta X}{2 \log X}.\] By further applications of the triangle inequality and $\int^{2X}_X \frac{dt}{t (\log t)^2} \le \frac{1}{\log X}$ it follows that 
\[ \sup_{Y \in [X, 2X]}\Big|  \sum_{n \le Y} \Lambda(n) e(n^m \theta) \Big| \ge \delta\eta X/16.\]
The desired conclusion \cref{q-desire-conclusion} now follows from \cref{log-free-weyl}.
\end{proof}

\subsection{Effectivity}\label{app-alt} The proof outline above for \cref{log-free-weyl} gives ineffective bounds due to the invocation of the Siegel--Walfisz theorem in Hua's work. However, one can replace this with a version of the prime number theorem in progressions incorporating an additional correction term for a potential Siegel zero such as \cite[Equation (5.71)]{IK-book}. Specifically, for $q \le e^{\sqrt{\log X}}$ and $(b,q) = 1$ we have
\[ \sum_{n \le X: n \equiv b \mdsub{q}} \Lambda(n) = \frac{X}{\phi(q)} - \frac{\overline{\chi(b)}}{\phi(q)}\frac{X^{\beta}}{\beta} + O(X e^{-c\sqrt{\log X}}),\] where here $\chi$ is some quadratic Dirichlet character for which $L(s,\chi)$ has a Siegel zero $\beta$. The Siegel zero term introduces a secondary main term in Hua's asymptotic formula, now of the form $-\frac{X^{\beta}}{\phi(q)} \tilde{S}(a,q) \tilde{\nu}(\eta X^m)$, where $\tilde S(a,q) = \sum_{b \mdsub{q}, (b,q) = 1} \overline{\chi(b)} e(ab^m/q)$ and $\tilde\nu(y) = \int^1_0 x^{\beta - 1} e(x^m y) dx$. These terms satisfy similar estimates to $S, \nu$ in the Hua analysis, allowing us to draw an analogous conclusion. 

\bibliographystyle{amsplain0}
\bibliography{main.bib}

\providecommand{\bysame}{\leavevmode\hbox to3em{\hrulefill}\thinspace}
\providecommand{\MR}{\relax\ifhmode\unskip\space\fi MR }
% \MRhref is called by the amsart/book/proc definition of \MR.
\providecommand{\MRhref}[2]{%
  \href{http://www.ams.org/mathscinet-getitem?mr=#1}{#2}
}
\providecommand{\href}[2]{#2}
\begin{thebibliography}{10}

\bibitem{AKS99}
Rudolf Ahlswede, Levon~H. Khachatrian, and Andr{\'{a}}s S{\'{a}}rk{\"{o}}zy,
  \emph{On the quotient sequence of sequences of integers}, Acta Arith.
  \textbf{91} (1999), 117--132.

\bibitem{Alw23}
Ryan Alweiss, \emph{Monochromatic sums and products over $\mathbf{Q}$},
  arXiv:2307.08901.

\bibitem{Alw24}
Ryan Alweiss, \emph{Monochromatic sums and products of polynomials}, Discrete
  Anal. (2024), Paper No. 5, 7.

\bibitem{Bow25}
Matt Bowen, \emph{Monochromatic products and sums in 2-colorings of
  {$\mathbf{N}$}}, Adv. Math. \textbf{462} (2025), Paper No. 110095, 17.

\bibitem{BS24}
Matt Bowen and Marcin Sabok, \emph{Monochromatic products and sums in the
  rationals}, Forum Math. Pi \textbf{12} (2024), Paper No. e17, 12.

\bibitem{CP20}
Jonathan Chapman and Sean Prendiville, \emph{On the {R}amsey number of the
  {B}rauer configuration}, Bull. Lond. Math. Soc. \textbf{52} (2020), 316--334.

\bibitem{CS17}
Karol Cwalina and Tomasz Schoen, \emph{Tight bounds on additive {R}amsey-type
  numbers}, J. Lond. Math. Soc. (2) \textbf{96} (2017), 601--620.

\bibitem{DE36}
Harold Davenport and Paul Erd{\H{o}}s, \emph{On the quotient sequence of
  sequences of integers}, Acta Arith. \textbf{2} (1936), 147--151.

\bibitem{Gow01}
W.~T. Gowers, \emph{A new proof of {S}zemer\'edi's theorem}, Geom. Funct. Anal.
  \textbf{11} (2001), 465--588.

\bibitem{GreOp}
Ben Green, \emph{100 open problems},
  \url{https://people.maths.ox.ac.uk/greenbj/papers/open-problems.pdf}.

\bibitem{Gre05}
Ben Green, \emph{Roth's theorem in the primes}, Ann. of Math. (2) \textbf{161}
  (2005), 1609--1636.

\bibitem{Gre25}
Ben Green, \emph{Waring's problem with restricted digits}, Compos. Math.
  \textbf{161} (2025), 341--364.

\bibitem{green-tao-selberg}
Ben Green and Terence Tao, \emph{Restriction theory of the {S}elberg sieve,
  with applications}, J. Th\'eor. Nombres Bordeaux \textbf{18} (2006),
  147--182.

\bibitem{GT08}
Ben Green and Terence Tao, \emph{Quadratic uniformity of the {M}\"obius
  function}, Ann. Inst. Fourier (Grenoble) \textbf{58} (2008), 1863--1935.

\bibitem{GT-mobius}
Ben Green and Terence Tao, \emph{The {M}\"obius function is strongly orthogonal
  to nilsequences}, Ann. of Math. (2) \textbf{175} (2012), 541--566.

\bibitem{GT14}
Ben Green and Terence Tao, \emph{On the quantitative distribution of polynomial
  nilsequences---erratum}, Ann. of Math. (2) \textbf{179} (2014), 1175--1183.

\bibitem{Har81}
Glyn Harman, \emph{Trigonometric sums over primes. {I}}, Mathematika
  \textbf{28} (1981), 249--254.

\bibitem{Hua38}
Loo-Keng Hua, \emph{Some results in the additive prime-number theory}, Quart.
  J. Math. Oxford Ser. (2) \textbf{9} (1938), 68--80.

\bibitem{IK-book}
Henryk Iwaniec and Emmanuel Kowalski, \emph{Analytic number theory}, American
  Mathematical Society Colloquium Publications, vol.~53, American Mathematical
  Society, Providence, RI, 2004.

\bibitem{Kou19}
Dimitris Koukoulopoulos, \emph{The distribution of prime numbers}, Graduate
  Studies in Mathematics, vol. 203, American Mathematical Society, Providence,
  RI, [2019] \copyright 2019.

\bibitem{Mor17}
Joel Moreira, \emph{Monochromatic sums and products in {$\mathbf{N}$}}, Ann. of
  Math. (2) \textbf{185} (2017), 1069--1090.

\bibitem{Pel20}
Sarah Peluse, \emph{Bounds for sets with no polynomial progressions}, Forum
  Math. Pi \textbf{8} (2020), e16, 55.

\bibitem{PP24}
Sarah Peluse and Sean Prendiville, \emph{Quantitative bounds in the nonlinear
  {R}oth theorem}, Invent. Math. \textbf{238} (2024), 865--903.

\bibitem{ramare-ruzsa}
Olivier Ramar\'e and Imre~Z. Ruzsa, \emph{Additive properties of dense subsets
  of sifted sequences}, J. Th\'eor. Nombres Bordeaux \textbf{13} (2001),
  559--581.

\bibitem{Ric25}
Florian~K Richter, \emph{Sums and products in sets of positive density},
  arXiv:2507.00515.

\bibitem{San20}
Tom Sanders, \emph{Bootstrapping partition regularity of linear systems}, Proc.
  Edinb. Math. Soc. (2) \textbf{63} (2020), 630--653.

\bibitem{Tao07}
Terence Tao, \emph{Structure and randomness in combinatorics}, 48th Annual IEEE
  Symposium on Foundations of Computer Science (FOCS'07), IEEE, 2007,
  pp.~3--15.

\bibitem{Tao24}
Terence Tao, \emph{Dense sets of natural numbers with unusually large least
  common multiples}, Integers \textbf{24} (2024), Paper No. A100, 18.

\bibitem{TT25}
Terence Tao and Joni Ter\"av\"ainen, \emph{Quantitative bounds for {G}owers
  uniformity of the {M}\"obius and von {M}angoldt functions}, J. Eur. Math.
  Soc. (JEMS) \textbf{27} (2025), 1321--1384.

\end{thebibliography}

\end{document}